\documentclass[10pt,reqno]{amsart}
\usepackage[a4paper,top=2.6cm, bottom=2.6cm,left=2.6cm, right=2.6cm,twoside]{geometry}

\usepackage{amsthm}
 \usepackage{amsmath}
 \usepackage{amsfonts}
 \usepackage{amssymb}
 \usepackage{amscd}
 \usepackage{url}
 \usepackage[T1]{fontenc}
 \usepackage[utf8x]{inputenc}
 \usepackage[english]{babel}
 \usepackage[mathscr]{eucal}
 \usepackage{lmodern}
 \usepackage{bera}
 \usepackage{float} 
 \usepackage{bold-extra}
 \usepackage{textcomp}
 \usepackage{graphicx}
 \usepackage{caption}
\usepackage{subcaption} 
\usepackage{multicol}
 \usepackage[dvipsnames,table]{xcolor}
 \usepackage{mathtools} 
 \usepackage{enumitem}
 \usepackage{tikz} 
 \usepackage{tikz-cd}
 \usepackage{wasysym}
 \usepackage{todonotes}
 \usepackage{fancybox}
 \usepackage{fancyhdr}
 \usepackage{breqn}
 \usepackage[nodisplayskipstretch]{setspace}
\usepackage{hyperref}
\hypersetup{
	colorlinks=true,
	linkcolor=blue,
	filecolor=darkmagenta,      
	urlcolor=cyan,
	citecolor=red
}

 \theoremstyle{definition}  
  \newtheorem{defn}{Definition}[section]

   \newtheorem{rmk}[defn]{Remark}

  \theoremstyle{plain}  
  \newtheorem{thm}[defn]{Theorem}
  \newtheorem{lem}[defn]{Lemma}
  \newtheorem{prop}[defn]{Proposition}
  \newtheorem{cor}[defn]{Corollary}

  \theoremstyle{remark}

 \renewcommand{\sf}[1]{\textsf{#1}}

 \newcommand{\mbb}[1]{\mathbb{#1}}
 \newcommand{\mcl}[1]{\mathcal{#1}}

 \newcommand{\msc}[1]{\mathscr{#1}}

 \newcommand{\norm}[1]{\left\lVert#1\right\rVert}

 \newcommand{\M}[1]{\mbb{M}_{#1}}

 \newcommand{\B}[1]{\msc{B}({#1})}

 \newcommand{\ranko}[2]{|{#1}\rangle\langle{#2}|}

 \newcommand{\ip}[1]{\langle#1\rangle}

 \newcommand{\ran}[1]{\sf{range}(#1)}
 
 \renewcommand{\ker}[1]{\sf{ker}(#1)}

 \newcommand{\Matrix}[1]{\begin{bmatrix}#1\end{bmatrix}}

 \DeclareMathOperator{\T}{\sf{T}}
 \DeclareMathOperator{\tr}{\sf{tr}}
 
 \DeclareMathOperator{\id}{\sf{id}}

 \DeclareMathOperator{\lspan}{\sf{span}}

 \numberwithin{equation}{section}
 \allowdisplaybreaks[4] 
 \setlist[enumerate]{font=\upshape,noitemsep, topsep=0pt} 
 \setlist[itemize]{noitemsep, topsep=0pt}

\makeatletter
\@namedef{subjclassname@2020}{\textup{2020} Mathematics Subject Classification}
\makeatother

\title{On the rank of extremal marginal states}

\author{Repana Devendra}
\address{Repana Devendra, Department of Mathematics, Indian Institute of Technology Bombay, Powai, Mumbai 400076, India. }
\email{r.deva1992@gmail.com, ma16d020@smail.iitm.ac.in}

\author{Pankaj Dey}
 \address{Pankaj Dey, Department of Mathematics, Indian Institute of Technology Bombay, Powai, Mumbai 400076, India.}
 \email{pankajdey2022@gmail.com, pankaj@math.iitb.ac.in}

\author{Santanu Dey}
\address{Santanu Dey, Department of Mathematics, Indian Institute of Technology Bombay, Powai, Mumbai 400076, India. }
 \email{santanudey@iitb.ac.in, dey@math.iitb.ac.in}

\begin{document}

\date{\today}

\begin{abstract}
Let $\rho_1$ and $\rho_2$ be two states on $\mathbb{C}^{d_1}$ and $\mathbb{C}^{d_2}$ respectively. The marginal state space, denoted by $\msc{C}(\rho_1,\rho_2)$, is the set of all states $\rho$ on $\mathbb{C}^{d_1}\otimes \mathbb{C}^{d_2}$ with partial traces $\rho_1, \rho_2$. K. R. Parthasarathy established that if $\rho$ is an extreme point of $\msc{C}(\rho_1,\rho_2)$, then the rank of $\rho$ does not exceed $\sqrt{d_1^2+d_2^2-1}$. Rudolph posed a question regarding the tightness of this bound. In 2010, Ohno gave an affirmative answer by providing examples in low-dimensional matrix algebras $\M{3}$ and $\M{4}$. This article aims to provide a positive answer to the Rudolph question in various matrix algebras.
Our approaches, to obtain the extremal marginal states with tight upper bound, are based on Choi-Jamio\l kowski isomorphism and tensor product of extreme points.

\end{abstract}



\maketitle


\section{Introduction}


Let $\mcl{H}_j$ denote a finite-dimensional Hilbert space that represents the quantum system $S_j$ for $j=1,2$. The combined quantum system $S_{12}$ is associated with the Hilbert space $\mcl{H}_1\otimes \mcl{H}_2$, which is the tensor product of the Hilbert spaces $\mcl{H}_1$ and $\mcl{H}_2$. A \emph{state} $\rho$ on $\mcl{H}$ is  a positive operator with unit trace. Let $\rho_j$ represent a state on $\mcl{H}_j$ for $j=1,2$. Consider the collection $\msc{C}(\rho_1,\rho_2)$, which includes all states $\rho$ in the coupled system $S_{12}$ such that the partial trace with respect to the quantum system $S_1$ yields $\rho_2$, and the partial trace with respect to the quantum system $S_2$ yields $\rho_1$. It is known that the set $\msc{C}(\rho_1,\rho_2)$ is a compact convex set. So, by the Krein-Milman theorem, it implies that this set $\msc{C}(\rho_1,\rho_2)$ is closed convex hull of its extreme points, which we denote as $\msc{E}(\rho_1,\rho_2)$. 

The necessary and sufficient conditions for a state to be an extreme point within the set  $\msc{C}(\rho_1,\rho_2)$ have been studied in \cite{LaSt93,krp98, Ru04}. One of the necessary conditions stipulates that the rank of the extreme point is constrained above by $\sqrt{d_1^2+d_2^2}$, where $d_j$ represents the dimension of the Hilbert space $\mcl{H}_j$ for $j=1,2$. Later in $2005$, K. R. Parthasarathy \cite{krp05} improved the upper bound of the rank to $\sqrt{d_1^2+d_2^2-1}$. Rudolph posed a question whether the upper bound on the rank given by K. R. Parthasarathy is sharp. That is, whether $\max\{\text{rank}(\rho): \rho \text{ is an extreme point of the set } \msc{C}(\rho_1,\rho_2)\}=\lfloor{\sqrt{d_1^2+d_2^2-1}}\rfloor$, with $\lfloor{a}\rfloor$ indicating the largest integer that does not exceed $a$.

It has been established in \cite{krp05} that when $d_1=d_2=2$, all extreme points of the set $\msc{C}(\rho_1, \rho_2)$ are pure, which is equivalent to having a rank of one. This proves that the upper bound on the rank of these extreme points is not sharp. In \cite{PrSa07}, it is conjectured that every extreme point of the set $\msc{C}(\frac{I_{d_1}}{d_1}, \frac{I_{d_2}}{d_2})$ is pure for all $d_1 = d_2 \geq 3$. However, Ohno \cite{Ohno10} provided counterexamples demonstrating the existence of extreme points that are not pure. Furthermore, he examined the sharpness of the upper bound in lower-dimensional matrix algebras, proving that for $d_1 = d_2 = 3$ and $d_1 = d_2 = 4$, the upper bound $\lfloor{\sqrt{d_1^2 + d_2^2 - 1}}\rfloor$ is indeed sharp. This indicates that there exist extreme points $\rho$ in the set $\msc{C}(\rho_1, \rho_2)$ with a rank of $\lfloor \sqrt{d_1^2 + d_2^2 - 1} \rfloor$ when $d_1 = d_2 = 3$ and $d_1 = d_2 = 4$. However, the question posed by Rudolph remains open for matrix algebras $\M{d}$ where $d\geq 5$.

Regarding the case when $d_1\neq d_2$, to the best of our knowledge, the only scenario that has been investigated yet is $d_1 = 2, d_2 = 3$. It is shown in \cite{KS21} that the upper bound on the rank given by K. R. Parthasarathy is sharp for $d_1 = 2, d_2 = 3$. That is, there exists an extreme point $\rho$ in the set $\msc{C}(\rho_1, \rho_2)$ with rank $\lfloor \sqrt{2^2 + 3^2 - 1} \rfloor=3$.

This article aims to demonstrate that the upper bound is indeed sharp for various cases. Most explicitly, we establish that $\max\{\text{rank}(\rho): \rho \text{ is an extreme point of the set } \msc{C}(\rho_1,\rho_2)\}=\lfloor{\sqrt{d_1^2+d_2^2-1}}\rfloor$ holds true for the following instances:
\begin{enumerate}[label=(\roman*)]
\item $d_1=d_2=5$.
\item $d_1=d_2=9$.
\item $d_1=d_2=12$.
\item $d_1=d_2=5k;\quad 3\leq k\leq 14$.
\item $d_1=3$ and $d_2=4$.
\item $d_1=2$ and $d_2\geq 4$.
\end{enumerate}
Additionally, we present an example of an extreme point with rank $d_1+1$ for the case where $(d_1,d_2)=(d_1,d_1+1)$, which is strictly less than $\lfloor{\sqrt{d_1^2+d_2^2-1}}\rfloor$ for all $d_1\geq 4$.

Our approach to prove the above cases is using Choi-Jamio\l kowski isomorphism. The structure of the article is as follows. In Section \ref{pre}, we recall a few definitions and preliminary results from the literature that are needed in this paper. We study tensor product of extreme points in Section \ref{tensor product}. This is the key idea that is used in subsequent sections. Section \ref{sharpness for d_1=d_2} demonstrates various classes of extremal marginal states that attain the bound given by K. R. Parthasarathy for the cases when $d_1=d_2$. Most explicitly, we study the cases when $(d_1,d_2)=(5,5), (9,9), (12,12),(5k,5k)$ where $3\leq k\leq 14$. Section \ref{section 5} contains the similar study for $d_1\neq d_2$. Most precisely, we discuss the cases when $(d_1,d_2)=(3,4), (2,d)$ where $d\geq 4$. This section also contains an example of an extreme point with rank $d+1$ for the case when $(d_1,d_2)=(d,d+1)$ with $d\geq 2$. We conclude this paper with appendix in Section \ref{appendix}.

\section{Preliminaries}\label{pre}

Let $\mbb{C}^d$ represent the $d$-dimensional complex Hilbert space, where the inner product is defined to be linear in the second variable and conjugate linear in the first. We denote $\B{ \mbb{C}^{d_1},\mbb{C}^{d_2}}$ as the space of all linear maps from $\mbb{C}^{d_1}$ to $\mbb{C}^{d_2}$. The standard orthonormal basis of $\mbb{C}^d$ is represented by the set $\{e_i^{(d)}: 1\leq i\leq d\}$. We define $\M{d_1\times d_2}$ as the algebra of complex matrices of size $d_1\times d_2$, and for convenience, we use $\M{d}$ to refer to $\M{d\times d}$. Each element $A$ in $\B{\mbb{C}^{d_1},\mbb{C}^{d_2}}$ can be associated with its $d_2\times d_1$ matrix corresponding to the standard bases of $\mbb{C}^{d_1}$ and $\mbb{C}^{d_2}$. Consequently, we interpret each matrix $A\in\M{d_1\times d_2}$ as representing a linear map from $\mbb{C}^{d_2}$ to $\mbb{C}^{d_1}$ and vice versa. The identity matrix in $\M{d}$ is denoted as $I=I_d$. Additionally, we define $\T=\T_d: \M{d}\to \M{d}$ as the transpose operation, $\tr=\tr_{d}: \M{d}\to \mbb{C}$ as the trace operation, and $\id=\id_d:\M{d}\to\M{d}$ as the identity operation. By writing $A=[a_{ij}]\in\M{d_1\times d_2}$, it indicates that $A$ is a complex matrix of dimensions $d_1\times d_2$, with the entry in the $(i,j)^{th}$ position represented by $a_{ij}$, for all indices satisfying $1\leq i\leq d_1$ and $1\leq j\leq d_2$. Furthermore, if $A=[a_{ij}]\in\M{d_1\times d_2}$, the adjoint of $A$, denoted as $A^*$, is a matrix of size $d_2\times d_1$. This adjoint is defined as the conjugate transpose of $A$, specifically expressed as $A^*=[b_{ij}]\in\M{d_2\times d_1}$, where each entry $b_{ij}$ is given by $\overline{a_{ji}}$ for all $1\leq i\leq d_2$ and $1\leq j\leq d_1$.

A matrix $A\in\M{d}$ is said to be \textit{Hermitian} if $A=A^*$ and a Hermitian matrix $A$ is called \textit{positive (semi-definite)} if $\ip{x,Ax}\geq 0$ for all $x\in \mbb{C}^d$. The collection of all positive matrices within $\M{d}$ is denoted by $\M{d}^+$. If $A\in\M{d}^+$, then we write $A\geq 0$. Let $x\in\mbb{C}^{d_1}$ and $y\in\mbb{C}^{d_2}$. We define the mapping $\ranko{x}{y}:\mbb{C}^{d_2}\to\mbb{C}^{d_1}$ by the expression $\ranko{x}{y}(z):=\ip{y,z}x$ for every $z\in\mbb{C}^{d_2}$. It is important to observe that if $x=(x_1,x_2,\ldots,x_{d_1})^{\T}$ and $y=(y_1,y_2,\ldots,y_{d_2})^{\T}$, then the matrix representation of $\ranko{x}{y}$ with respect to the standard basis is given by $[x_i\overline{y_j}]=xy^*\in\M{d_1\times d_2}$. For convenience, we also refer to this matrix as $\ranko{x}{y}$. It is straightforward to verify that $A\ranko{x}{y}B=\ranko{Ax}{B^*y}$ holds for any $x\in\mbb{C}^{d_1}$, $y\in\mbb{C}^{d_2}$, and matrices $A\in\M{d_1}, B\in\M{d_2}$. Additionally, it is noteworthy that the matrix $\ranko{x}{x}$ is positive for all $x\in\mbb{C}^{d}$.

Let $\mbb{C}^{d_1}\otimes\mbb{C}^{d_2}$ denote the tensor product of the Hilbert spaces $\mbb{C}^{d_1}$ and $\mbb{C}^{d_2}$. For operators $A\in\B{\mbb{C}^{d_1}}$ and $B\in\B{\mbb{C}^{d_2}}$, there exists a unique operator $A\otimes B\in\B{\mbb{C}^{d_1}\otimes \mbb{C}^{d_2}}$ such that for all vectors $x\in\mbb{C}^{d_1}$ and $y\in\mbb{C}^{d_2}$, the relation $A\otimes B(x\otimes y)=Ax\otimes By$ holds. If $\{u_j\}_{j=1}^{d_1}$ and $\{v_j\}_{j=1}^{d_2}$ represent orthonormal bases for $\mbb{C}^{d_1}$ and $\mbb{C}^{d_2}$, respectively, then the set $\{u_i\otimes v_j\}_{i,j=1}^{d_1,d_2}$ forms an orthonormal basis for $\mbb{C}^{d_1}\otimes\mbb{C}^{d_2}$. We arrange the elements of the set $\{u_i\otimes v_j :1\leq i\leq d_1, 1\leq j\leq d_2\}$ in the lexicographic order, defined by the condition $(i, j) < (i',j')$ if $i < i'$ or if $i = i'$ and $j < j'$. If $A=[a_{ij}]\in\M{d_1}$ and $B=[b_{ij}]\in\M{d_2}$, then according to this ordering, we have $A\otimes B=[a_{ij}B]$. This identification allows us to view $\M{d_1}\otimes\M{d_2}$ as $\M{d_1}(\M{d_2})$, treating the elements of $\M{d_1}\otimes\M{d_2}$ as block matrices, where each block corresponds to an entry from $A$ multiplied by the entire matrix $B$.

Given the linear maps $\Phi_1:\M{d_1}\to\M{d_2}$ and $\Phi_2:\M{d_3}\to\M{d_4}$, the tensor product map $\Phi_1\otimes\Phi_2:\M{d_1}\otimes\M{d_3}\to\M{d_2}\otimes\M{d_4}$ is uniquely determined by the property that $\Phi_1\otimes\Phi_2(A\otimes B)=\Phi_1(A)\otimes\Phi_2(B)$ for all elements $A\in\M{d_1}$ and $B\in\M{d_3}$.

\begin{defn}
	A linear  map $\Phi:\M{d_1}\to\M{d_2}$ is said to be 
	\begin{enumerate}[label=(\roman*)]
		\item \emph{Hermitian preserving} if $\Phi(X^*)=\Phi(X)^*$ for all $X\in\M{d_1}$.
		\item \emph{positive} if $\Phi(\M{d_1}^+)\subseteq \M{d_2}^+$.
		\item \emph{completely positive (CP)} if the map $\id_k\otimes\Phi:\M{k}\otimes\M{d_1}\to\M{k}\otimes\M{d_2}$ is a positive map for all $k\geq 1$.
		\item \emph{unital} if $\Phi(I_{d_1})=I_{d_2}$.
		\item \emph{trace-preserving } if $\tr(\Phi(X))=\tr(X)$ for all $X\in\M{d_1}$.
	\end{enumerate}
\end{defn}

The Hilbert-Schmidt inner product on the space $\M{d}$ is defined as $\ip{X, Y}:=\tr(X^*Y)$. This inner product enables the consideration of $\M{d}$ as a Hilbert space. Utilizing this inner product, we can define the dual $\Phi^*:\M{d_2}\to\M{d_1}$ of a linear map $\Phi:\M{d_1}\to\M{d_2}$ as the unique linear map that fulfills the condition $\tr({\Phi(X)}^*Y)=\tr(X^*\Phi^*(Y))$ for every $X\in\M{d_1}$ and $Y\in\M{d_2}$. Let $\B{\M{d_1},\M{d_2}}$ denote the space of all linear maps from $\M{d_1}$ to $\M{d_2}$.

In the realm of quantum information theory, the Choi-Jamio\l kowski isomorphism \cite{Cho75, Jam72}, commonly known as channel-state duality, illustrates the connection between quantum channels, characterized by completely positive maps, and quantum states, which are expressed through density matrices.

\begin{thm}[Choi-Jamio\l kowski \cite{Cho75, Jam72}]\label{thm-C-J-map}
	Let $\msc{J}$ be the linear map from $\B{\M{d_1},\M{d_2}}\to\M{d_1}\otimes\M{d_2}$ defined by 
	$$\Phi\mapsto (\id_{d_1}\otimes \Phi)(\ranko{\Omega_{d_1}}{\Omega_{d_1}}),$$
	where $\Omega_{d_1}=\sum_{j=1}^{d_1}e_j^{(d_1)}\otimes e_j^{(d_1)}$. Then we have the following:
	
	\begin{enumerate}
		\item $\msc{J}$ is a vector space isomorphism. 
		\item $\Phi$ is Hermitian preserving if and only if $\msc{J}(\Phi)$ is Hermitian.
		\item $\Phi$ is positive map if and only if $\msc{J}(\Phi)$ is positive on simple tensor vectors, that is, $\ip{x\otimes y, \msc{J}(\Phi)x\otimes y}\geq 0$ for all $x\in\mbb{C}^{d_1}$ and $y\in\mbb{C}^{d_2}$. 
		\item $\Phi$ is CP if and only if $\msc{J}(\Phi)$ is positive. 
	\end{enumerate}
\end{thm}

\begin{cor}[Choi-Kraus \cite{Cho75,Kra71}]\label{cor-choi-kraus-rep}
	Let $\Phi:\M{d_1}\to\M{d_2}$ be a linear map. Then $\Phi$ is CP if and only if there exists $\{V_j\}_{j=1}^n\subseteq\M{d_1\times d_2}$ such that  
	\begin{align}\label{eq-Kraus-decompo}
		\Phi=\sum_{j=1}^n\mathrm{Ad}_{V_j}
	\end{align}
	where $\mathrm{Ad}_V(X):=V^*XV$ for all $X\in\M{d_1}$.
\end{cor}

The representation given by \eqref{eq-Kraus-decompo} is known as the \emph{Kraus decomposition} of the map $\Phi$, with the operators $V_i$ being called Kraus operators. It is noteworthy that the Kraus decomposition of $\Phi$ need not be unique. However, in cases where $\Phi$ is expressed as $\sum_{j=1}^n\mathrm{Ad}_{V_j}$ and also as $\sum_{i=1}^m\mathrm{Ad}_{W_i}$, which are two different Kraus decompositions, it can be concluded that $\lspan\{V_j:1\leq j\leq n\}=\lspan\{W_i: 1\leq i\leq m\}$.

\begin{defn}
	Let $\Phi:\M{d_1}\to\M{d_2}$ be a CP map. The smallest number of Kraus operators necessary to express $\Phi$ as indicated in \eqref{eq-Kraus-decompo} is referred to as the \emph{Choi-rank} of $\Phi$, denoted by $CR(\Phi)$. A Kraus decomposition of $\Phi$, as outlined in \eqref{eq-Kraus-decompo}, is considered minimal if $n=CR(\Phi)$.
\end{defn}

Let $\Phi:\M{d_1}\to\M{d_2}$ be a CP map with Kraus operators $\{V_j:1\leq j\leq n\}$. It can be checked that $n=CR(\Phi)$ if and only if $\{V_j:1\leq j\leq n\}$ is linearly independent. The following proposition connects the Choi-rank of $\Phi$ with $\mbox{rank}\big(\msc{J}(\Phi)\big)$.

\begin{prop}[\cite{Cho75}]\label{prop-choi-rank}
	Let $\Phi:\M{d_1}\to\M{d_2}$ be a CP map. Then $CR(\Phi)=\mbox{rank}\big(\msc{J}(\Phi)\big)$.
\end{prop}

Let $\{u_j:1\leq j\leq d_1\}$ and $\{v_j:1\leq j\leq d_2\}$ be orthonormal basis of $\mbb{C}^{d_1}$ and $\mbb{C}^{d_2}$, respectively. For any matrix $Z\in\M{d_1}\otimes\M{d_2}$, we define the sesquilinear forms as follows:

\begin{align*}
	&\Delta_1(x_1,y_1)=\sum_{j=1}^{d_2}\ip{x_1\otimes v_j, Z(y_1\otimes v_j)} \qquad \text{for all } x_1,y_1\in\mbb{C}^{d_1},\\
	&\Delta_2(x_2,y_2)=\sum_{j=1}^{d_1}\ip{u_j\otimes x_2, Z(u_j\otimes y_2)}\qquad \text{for all } x_2,y_2\in\mbb{C}^{d_2}.
\end{align*}

From the fact that any sesquilinear form is given by a matrix, it follows that there exists unique $Z_j\in\M{d_j}$ such that 
\begin{align*}
	\Delta_j(x_j,y_j)=\ip{x_j, Z_jy_j} \quad \text{for all } x_j,y_j\in\mbb{C}^{d_j}, j=1,2.
\end{align*}
It is important to note that the operators $Z_j$ do not depend on the specific orthonormal bases $\{u_j\}$ and $\{v_j\}$. Furthermore, $Z_1$ can be derived from $Z$ by performing a trace over $\mbb{C}^{d_2}$. Consequently, we have a mapping from $\M{d_1}\otimes\M{d_2}$ to $\M{d_1}$ defined by
$$Z\mapsto\tr_{\mbb{C}^{d_2}}(Z)$$
which have the following properties:
\begin{enumerate}[label=(\roman*)]
	\item $\tr_{\mbb{C}^{d_2}}$ is a linear map preserves the both adjoint and positivity. In fact, $\tr_{\mbb{C}^{d_2}}$ is a CP map.
	\item $\tr(Z)=\tr\big(\tr_{\mbb{C}^{d_2}}(Z)\big)$ for all $Z\in\M{d_1}\otimes\M{d_2}$.
	\item $\tr_{\mbb{C}^{d_2}}((A\otimes I_{d_2})Z(B\otimes I_{d_2}))=A\tr_{\mbb{C}^{d_2}}(Z)B,$ for all $A,B\in\M{d_1}$ and $Z\in\M{d_1}\otimes\M{d_2}$. 
\end{enumerate}
 In particular, the equations $\tr_{\mbb{C}^{d_2}}(A_1\otimes A_2)=\tr(A_2)A_1$ and $\tr(A_1\otimes A_2)=\tr(A_1)\tr(A_2)$ are valid for all matrices $A_j\in\M{d_j}$, where $j=1,2$. For notational simplicity, we, henceforth, denote $\tr_{\mbb{C}^{d_j}}$ as $\tr_j$ for $j=1,2$.

 If $\rho$ is a state within $\M{d_1}\otimes\M{d_2}$, then both $\tr_2(\rho)$ and $\tr_1(\rho)$ are also states in $\M{d_1}$ and $\M{d_2}$, respectively. Given a state $\rho_j\in\M{d_j}$ for $j=1,2$. Consider the set 
 \begin{align*}
 	\msc{C}(\rho_1,\rho_2):=\{\rho\in(\M{d_1}\otimes\M{d_2})^+: \tr_2(\rho)=\rho_1, \tr_1(\rho)=\rho_2 \mbox{ and } \tr(\rho)=1\}.
 \end{align*}

 \begin{prop}
 	Let $\rho_j\in\M{d_j}$  be states for $j=1,2$. Then the set   $\msc{C}(\rho_1,\rho_2)$ is compact convex set.
 \end{prop}

 \begin{proof}
 	Indeed, the set $\msc{C}(\rho_1,\rho_2)$ is convex. Since, for any state $\rho$,  $\norm{\rho} \leq \tr(\rho) = 1$, the set $\msc{C}(\rho_1,\rho_2)$ is bounded. To establish its compactness, it suffices to show that the set is closed. Consider a sequence $Z_n \in \msc{C}(\rho_1,\rho_2)$ such that $Z_n \to \rho \in \M{d_1} \otimes \M{d_2}$. Clearly, $\rho\geq 0$ since the set $(\M{d_1} \otimes \M{d_2})^+$ is closed. Moreover, we have $\rho_j = \tr_j(Z_n) \to \tr_j(\rho)$ for $j=1,2$ and $1=\tr(Z_n) \to \tr(\rho)$. Therefore, $\tr_2(\rho) = \rho_1$ and $\tr_1(\rho) = \rho_2$ and, hence $\rho\in\msc{C}(\rho_1,\rho_2)$. This completes the proof.
 \end{proof}

 This is to be noted that, to study the extreme points of $\msc{C}(\rho_1,\rho_2)$, without loss of generality, we may assume both $\rho_1$ and $\rho_2$ are invertible because of the following reason. For any operator $A$ defined on a Hilbert space $\mcl{H}$, let $\ker{A}$ denote the null space and $\ran{A}$ represent the range space of $A$. Let $\rho\in\msc{C}(\rho_1,\rho_2)$.
 Suppose $x\in\ker{\rho_1}$, then we have
 \begin{align*}
 	0&=\ip{x,\rho_1x}
 	=\sum_{j=1}^{d_2}\ip{x\otimes v_j,\rho (x\otimes v_j)},
 \end{align*}
 where $\{ v_j\}_{j=1}^{d_2}$ constitutes any orthonormal basis of $\mathbb{C}^{d_2}$. Given that $\rho$ is a positive operator, it follows that $\rho(x\otimes v_j)=0$ for all $j$. Consequently, we deduce that $\ker{\rho_1}\otimes \mathbb{C}^{d_2}\subseteq\ker{\rho}$. In a similar manner, it can be shown that $\mathbb{C}^{d_1}\otimes \ker{\rho_2}\subseteq\ker{\rho}$. By considering the orthogonal complements, we arrive at the conclusion that $\ran{\rho}\subseteq\ran{\rho_1}\otimes\ran{\rho_2}$.  Denote $\mathcal{H}_0=\ran{\rho_1}\otimes\ran{\rho_2}$. With respect to the decomposition $\mathbb{C}^{d_1}\otimes\mathbb{C}^{d_2}=\mathcal{H}_0\oplus\mathcal{H}_0^\perp$, we obtain
 \begin{align*}
 	\rho=\begin{pmatrix}
 		\widetilde{\rho} & 0 \\
 		0 & 0\\
 	\end{pmatrix}.
 \end{align*}
 So $\rho$ is an extreme point of $\msc{C}(\rho_1,\rho_2)$ if and only if $\tilde{\rho}$ is an extreme point of $\msc{C}(\widetilde{\rho_1},\widetilde{\rho_2})$ where $\widetilde{\rho_1}=\tr_{\ran{\rho_2}}(\tilde{\rho})$ and $\widetilde{\rho_2}=\tr_{\ran{\rho_1}}(\tilde{\rho})$. Also, with respect to the decomposition $\mathbb{C}^{d_j}=\ran{\rho_j}\oplus\ker{\rho_j}$, we get
 \begin{align*}
 	\rho_j=\begin{pmatrix}
 		\widetilde{\rho_j}& 0 \\
 		0 & 0\\
 	\end{pmatrix}
\end{align*}
 for $j=1,2$. Clearly $\widetilde{\rho_j}$ is invertible for $j=1,2$. Therefore, to study the extreme points of $\msc{C}(\rho_1,\rho_2)$, without loss of generality, we may assume both $\rho_1$ and $\rho_2$ are invertible.

 We will now utilize the Choi-Jamio\l kowski map $\msc{J}$, to convert the set $\msc{C}(\rho_1,\rho_2)$ into a subset of $\B{\M{d_1},\M{d_2}}$. The linearity of $\msc{J}$, which implies continuity, guarantees that it will map a compact convex set to another compact convex set. Additionally, this mapping retains the characteristics of extreme points, that is, extreme points are transformed into extreme points and vice versa.

 \begin{prop}\label{prop-C-J-extreme to extreme}
 Let $\msc{J}:\B{\M{d_1},\M{d_2}}\to\M{d_1}\otimes\M{d_2}$ be the Choi-Jamio\l kowski map given in the Theorem \ref{thm-C-J-map}. Then
 \begin{enumerate}[label=(\roman*)]
 	\item $\msc{J}^{-1}\big(\msc{C}(\rho_1,\rho_2)\big)=\{\Phi:\M{d_1}\to\M{d_2}: \Phi \mbox{ is CP}, \; \Phi(I_{d_1})=\rho_2 \text{ and } \Phi^*(I_{d_2})=\rho_1^{\T}\}$ is a compact convex set.
 	\item $\rho$ is an extreme point of the set $\msc{C}(\rho_1,\rho_2)$ if and only if $\msc{J}^{-1}(\rho)$ is an extreme point of the set $\msc{J}^{-1}\big(\msc{C}(\rho_1,\rho_2)\big)$.
 	\item $\sup\{\mbox{rank}(\rho): \rho \in\msc{E}(\rho_1,\rho_2)\}=\sup\{CR(\Phi): \Phi \in\msc{J}^{-1}\big(\msc{E}(\rho_1,\rho_2)\big)\}$, where $\msc{E}(\rho_1,\rho_2)$ denotes extreme points of the set $\msc{C}(\rho_1,\rho_2)$.
 \end{enumerate}
 \end{prop}

 \begin{proof}
 	$(i)$ Let $\Phi\in\B{\M{d_1},\M{d_2}}$. Then, 
 	\footnotesize
 	\begin{align*}
 		\tr_{1}(\msc{J}(\Phi))&=\sum_{i,j=1}^{d_1}\tr\left(\ranko{e_i^{(d_1)}}{e_j^{(d_1)}}\right)\Phi\left(\ranko{e_i^{(d_1)}}{e_j^{(d_1)}}\right)=\sum_{i=1}^{d_1}\Phi\left(\ranko{e_i^{(d_1)}}{e_i^{(d_1)}}\right)=\Phi(I_{d_1}),\\
 		\tr_2(\msc{J}(\Phi))&=\sum_{i,j=1}^{d_1}\ranko{e_i^{(d_1)}}{e_j^{(d_1)}}\tr\left(\Phi\left(\ranko{e_i^{(d_1)}}{e_j^{(d_1)}}\right)\right)=\sum_{i,j=1}^{d_1}\tr\left(\Phi^*(I_{d_2})\ranko{e_i^{(d_1)}}{e_j^{(d_1)}}\right)\ranko{e_i^{(d_1)}}{e_j^{(d_1)}}=(\Phi^*(I_{d_2})^*)^{\T}
 	\end{align*}
 	\normalsize
 	So $\msc{J}\Big(\left\{\Phi:\M{d_1}\to\M{d_2}: \Phi \mbox{ is CP}, \; \Phi(I_{d_1})=\rho_2 \text{ and }\Phi^*(I_{d_2})=\rho_1^{\T}\right\}\Big)=\msc{C}(\rho_1,\rho_2)$. Indeed, $\msc{J}^{-1}\big(\msc{C}(\rho_1,\rho_2)\big)$  is convex as $\msc{J}$ and $\msc{J}^{-1}$ are linear.  Since the map $\msc{J}$ and $\msc{J}^{-1}$ are continuous, we have that $\msc{J}^{-1}\big(\msc{C}(\rho_1,\rho_2)\big)$ is compact.   \\
 	$(ii)$ Follows trivially.\\
 	$(iii)$ Follows from  the Proposition \ref{prop-choi-rank}.
 \end{proof}

 Let $A_j\in\M{d_j}$ for $j=1,2$. Denote
 \begin{align*}
 	\msc{CP}[\M{d_1},\M{d_2}; A_1,A_2]:=&\left\{\Phi:\M{d_1}\to\M{d_2}: \Phi \mbox{ is CP},\; \Phi(I_{d_1})=A_2 \text{ and } \Phi^*(I_{d_2})=A_1\right\},\\
 	\msc{CP}[\M{d_1},\M{d_2}; A_2]:=&\left\{\Phi:\M{d_1}\to\M{d_2}: \Phi \mbox{ is CP and }\Phi(I_{d_1})=A_2\right\},\\
 	\msc{MR}(A_1,A_2):=&\sup\left\{CR(\Phi): \Phi \text{ is an extreme point of }\msc{CP}[\M{d_1},\M{d_2}; A_1, A_2]\right\}.
 \end{align*}

 In view of Proposition \ref{prop-C-J-extreme to extreme}, it follows that the above sets $ \msc{CP}[\M{d_1},\M{d_2}; A_1, A_2]$ and $\msc{CP}[\M{d_1},\M{d_2}; A_2]$ are compact convex sets, and studying extreme points of $\msc{C}(\rho_1, \rho_2)$ is equivalent with studying extreme points of $\msc{CP}[\M{d_1},\M{d_2}; \rho_1^{\T}, \rho_2]$. The following results provide necessary and sufficient conditions for an element within the above sets to be extreme for these sets. The reader may refer \cite{Cho75, krp98, krp05, Ru04} for comprehensive insights on this subject.

  \begin{thm}[Thoerem 5, \cite{Cho75}]\label{choi}
 	Let $A_2\in\M{d_2}$ and $\Phi\in \msc{CP}[\M{d_1},\M{d_2}; A_2]$. Then $\Phi$ is an extreme point of the set $ \msc{CP}[\M{d_1},\M{d_2}; A_2]$ if and only if there exists a Kraus decompositon  of $\Phi=\sum_{j=1}^n\mathrm{Ad}_{V_j}$, such that
 	the set $\{V_i^*V_j\}_{i,j=1}^{n}\subseteq\M{d_2}$ is linearly independent.
 \end{thm}

 \begin{rmk}\label{remark}
 	Let $\Phi\in \msc{CP}[\M{d_1},\M{d_2}; A_2]$ with minimal Kraus decomposition $\Phi=\sum_{j=1}^n\mathrm{Ad}_{V_j}$. Then $\Phi$ is an extreme point of the set $ \msc{CP}[\M{d_1},\M{d_2}; A_2]$ if and only if the set $\{V_i^*V_j\}_{i,j=1}^{n}\subseteq\M{d_2}$ is linearly independent. The reader may refer to the proof of the above theorem for this.
 \end{rmk}

 \begin{defn}
 	Let $V, W$ be two vector spaces. A pair of sets $\{v_j\}_{j=1}^{n}\subseteq V$ and $\{w_j\}_{j=1}^{n}\subseteq W$ is said to be \emph{bi-linearly independent} if the set $\{(v_j,w_j)\}_{j=1}^{n}$ is linearly independent in the vector space $V\times W$. i.e., whenever $\lambda_j\in\mbb{C}$ such that $\sum_{j=1}^n\lambda_jv_j=0$ and $\sum_{j=1}^n\lambda_jw_j=0$ implies $\lambda_j=0$ for all $j$.
 \end{defn}

 \begin{rmk}
 	From the definition of bi-linearly independent, it follows that if either $\{v_j\}_{j=1}^{n}$ or $\{w_j\}_{j=1}^{n}$ is linearly independent, then the pair  $\{v_j\}_{j=1}^{n}$ and $\{w_j\}_{j=1}^{n}$ is bi-linearly independent. But the converse need not be true in general.
 \end{rmk}

 Since $\msc{CP}[\M{d_1},\M{d_2}; A_1, A_2]\subseteq \msc{CP}[\M{d_1},\M{d_2}; A_2]$, if $\Phi\in\msc{CP}[\M{d_1},\M{d_2}; A_1, A_2]$ be an element such that it is an extreme point of the set $\msc{CP}[\M{d_1},\M{d_2}; A_2]$, then it is also an extreme point of the set $\msc{CP}[\M{d_1},\M{d_2}; A_1, A_2]$. However, the reverse implication is not necessarily true. The subsequent theorem provides a characterization of the extreme points within the set $\msc{CP}[\M{d_1},\M{d_2}; A_1, A_2]$.

 \begin{thm}[\cite{LaSt93, krp98, Ru04}]\label{thm-extr-condition-for doubly}
 	Let $A_j\in\M{d_j}$ for $j=1,2$ and $\Phi\in \msc{CP}[\M{d_1},\M{d_2}; A_1,A_2]$. Then $\Phi$ is an extreme point of the set $ \msc{CP}[\M{d_1},\M{d_2}; A_1,A_2]$ if and only if there exists $\{V_j\}_{j=1}^n\subseteq \M{d_1\times d_2}$ such that $\Phi=\sum_{i=j}^n\mathrm{Ad}_{V_j}$ and the pair
 	$\{V_i^*V_j\}_{i,j=1}^{n}\subseteq\M{d_2}$ and $\{V_jV_i^*\}_{i,j=1}^{n}\subseteq\M{d_1}$ is bi-linearly independent.
 \end{thm}

 \begin{rmk}
 	If the pair of sets $\{V_i^*V_j\}_{i,j=1}^{n}\subseteq\M{d_2}$ and $\{V_jV_i^*\}_{i,j=1}^{n}\subseteq\M{d_1}$ is bi-linearly independent, it follows that the collection $\{V_j\}_{j=1}^n\subseteq\M{d_1\times d_2}$ is also linearly independent. To illustrate this, assume that $\{V_j\}_{j=1}^n$ is linearly dependent. This means there exists a non-zero vector $(\lambda_1,\lambda_2,\hdots,\lambda_n)\in\mbb{C}^n$ such that $\sum_{j=1}^n\lambda_jV_j=0$. Consequently, for each $1\leq i\leq n$, we have $V_i^*\left(\sum_{j=1}^n\lambda_jV_j\right)=0$ and $\left(\sum_{j=1}^n\lambda_jV_j\right)V_i^*=0$. Therefore, it follows that $\sum_{i,j=1}^n\lambda_j V_i^*V_j=0$ and $\sum_{i,j=1}^n\lambda_jV_jV_i^*=0$. This indicates that the sets $\{V_i^*V_j\}_{i,j=1}^{n}\subseteq\M{d_2}$ and $\{V_jV_i^*\}_{i,j=1}^{n}\subseteq\M{d_1}$ cannot be bi-linearly independent.
 \end{rmk}

  Let $A_j\in\M{d_j}$ for $j=1,2$. In view of Theorem \ref{thm-extr-condition-for doubly}, we have $\msc{MR}(A_1,A_2)\leq \lfloor{\sqrt{d_1^2+d_2^2}}\rfloor$. KRP \cite{krp05} improved this upper bound on the rank of the extreme points of $\msc{CP}[\M{d_1},\M{d_2}; A_1,A_2]$, and established that $\msc{MR}(A_1,A_2)\leq \lfloor{\sqrt{d_1^2+d_2^2-1}}\rfloor$. Furthermore, he proved that all extreme points of the set $\msc{CP}[\M{2},\M{2}; I_2, I_2]$ are pure, that is, $\msc{MR}(\frac{I_2}{2},\frac{I_2}{2})=1$. This shows that the upper bound given by K. R. Parthasarathy is not sharp for $d_1=d_2=2$.

  It is conjectured in \cite{PrSa07} that every extreme point of the set $\msc{CP}[\M{d},\M{d}; I_d, I_d]$ is pure for all $d\geq 3$. However, Ohno \cite{Ohno10} provided counterexamples demonstrating the existence of extreme points that are not pure. These examples also assert that for any dimension $d\geq 3$, there exists an extreme point in the set $\msc{CP}(\M{d},\M{d}; I_{d}, I_{d})$ with Choi-rank $d$. Most explicitly, he obtained the following.

  Let us denote the standard matrix units of $\M{d}$ by $E_{ij}^{(d)}=\ranko{e_i^{(d)}}{e_j^{(d)}}$. For the sake of simplicity, it is common practice to exclude the superscript $(d)$, thereby referring to them as $E_{ij}$.

 \begin{thm}[Theorem 2.8, \cite{Ohno10}]\label{thm-extrem-rank-lowerbound}
	Given $d\geq 3$, define the map $\Phi:\M{d}\to\M{d}$ by
	\begin{align}\label{eq-extreme-rank-lowerbound-2}
		\Phi=\sum_{j=1}^{d}\mathrm{Ad}_{V_j},
	\end{align}
	where
	\begin{equation}\label{eq-extreme-rank-lowerbound}
		\begin{cases}
			V_1&=\sqrt{\frac{d-2}{d-1}}\sum\limits_{j=2}^{d}E_{jj},\\
			V_k&=\frac{1}{\sqrt{d-1}}(E_{1k}+E_{k1}), \text{ for all } 2\leq k\leq d.
		\end{cases}
	\end{equation}
	Then $\Phi\in\msc{CP}(\M{d},\M{d}; I_{d}, I_{d})$ is an extreme point of the set $\msc{CP}(\M{d},\M{d}; I_{d})$ with $CR(\Phi)=d$ (hence extreme point of $\msc{CP}(\M{d},\M{d}; I_{d}, I_{d})$).
 \end{thm}

\begin{rmk}
	Given $d\geq 3$, in view of Theorem \ref{thm-extrem-rank-lowerbound}, we have $\msc{MR}(\frac{I_d}{d},\frac{I_d}{d})\geq d$. It is also noteworthy that the Kraus operators of the extreme point $\Phi$, given in Theorem \ref{thm-extrem-rank-lowerbound}, are Hermitian.
\end{rmk}

   Given $A_j\in\M{d_j}$ for $j=1,2$, Rudolph \cite{Ru04} posed a question regarding the tightness of the bound given by K. R. Parthasarathy for $d_1 = d_2 \geq 3$, that is, whether $\msc{MR}(A_1, A_2)=\lfloor{\sqrt{d_1^2 + d_2^2 - 1}}\rfloor$ for $d_1 = d_2 \geq 3$. In $2010$, Ohno \cite{Ohno10} proved that the upper bound given by K. R. Parthasarathy is indeed sharp for $d_1 = d_2 = 3$ and $d_1 = d_2 = 4$. Most precisely, he obtained the following results.

   \begin{thm}[Theorem 2.6, \cite{Ohno10}]\label{thm-extreme-rank-4}
   	Let $\Phi:\M{3}\to\M{3}$ be a CP map defined by 
   	\begin{align}\label{eq-extreme-rank-4-2}
   		\Phi(X)=\frac{1}{4}\sum_{j=1}^{4}W_j^*XW_j=\frac{1}{4}\Matrix{x_{11}+2x_{22}+x_{33}&\sqrt{6}x_{23}&\sqrt{2}x_{31}\\ \sqrt{6}x_{32}&x_{11}+3x_{33}&\sqrt{2}x_{12}\\ \sqrt{2}x_{13}&\sqrt{2}x_{21}&2x_{11}+2x_{22}},
   	\end{align}
   	for all $X=[x_{ij}]\in\M{3}$, where
   	\small
   	\begin{equation}\label{eq-extreme-rank-4}
   		\begin{cases}
   			W_1&=E_{11}, \\
   			W_2&=E_{12}+\sqrt{2}E_{23},\\
   			W_3&=\sqrt{2}E_{21}+\sqrt{3}E_{32},\\
   			W_4&=E_{31}+\sqrt{2}E_{13}.  
   		\end{cases}
   	\end{equation}
   	\normalsize
   	Then $\Phi$ is an extreme point of the set $\msc{CP}(\M{3},\M{3}; I_{3}, I_{3})$ with $CR(\Phi)=4$.
   \end{thm}

    \begin{thm}[Theorem 2.7, \cite{Ohno10}]\label{thm-extreme-rank-5}
	Let $\Phi:\M{4}\to\M{4}$ be a CP map defined by 
  	\small
  	\begin{align}
  		\Phi(X)=\frac{1}{5}\sum_{j=1}^{4}W_j^*XW_j=\frac{1}{4}\Matrix{x_{22}+3x_{33}&\sqrt{2}x_{24}&0&\sqrt{6}x_{31}\\ \sqrt{2}x_{42}&x_{11}+x_{33}+2x_{44}&x_{12}+x_{31}&0\\ 0&x_{13}+x_{21}&x_{11}+x_{22}+2x_{44}&2x_{42}\\\sqrt{6}x_{13}&0&2x_{24}&2x_{11}+2x_{22}},
  	\end{align}
	\normalsize
	for all $X=[x_{ij}]\in\M{4}$, where
   	\small
   	\begin{equation}\label{eq-extreme-rank-5}
  		\begin{cases}
  			W_1&=E_{13}+E_{32}, \\
   			W_2&=\sqrt{2}E_{24}+\sqrt{2}E_{43},\\
  			W_3&=\sqrt{2}E_{14}+\sqrt{3}E_{31},\\
  			W_4&=E_{21}+\sqrt{2}E_{42}, \\
  			W_5&=E_{12}+E_{23}
 		\end{cases}
	\end{equation}
	\normalsize
	Then $\Phi$ is an extreme point of the set $\msc{CP}(\M{4},\M{4}; I_{4}, I_{4})$ with $CR(\Phi)=5$.
 \end{thm}

   \begin{rmk}
   	In view of Theorem \ref{thm-extreme-rank-4}, \ref{thm-extreme-rank-5}, we have $\msc{MR}(\frac{I_3}{3},\frac{I_3}{3})=4=\lfloor{\sqrt{3^2+3^2-1}}\rfloor$ and $\msc{MR}(\frac{I_4}{4},\frac{I_4}{4})=5=\lfloor{\sqrt{4^2+4^2-1}}\rfloor$, respectively.
   \end{rmk}

   \section{Tensor product of extreme points}\label{tensor product}

   In this section, we examine the conditions under which the tensor product of two extreme points is an extreme point. We start with a straightforward proposition that is readily apparent. For thoroughness, we include the details.

   \begin{prop}\label{prop-LI-functionals}
   	Let $V$ be a finite-dimensional vector space, and $\{v_i\}_{i=1}^n$ be any subset. Then the subset $\{v_i\}_{i=1}^n$ is linearly independent if and only if there exists a set of linear functionals $\{f_j\}_{j=1}^n$ on $V$ such that $f_j(v_i)=\delta_{ij}$ for all $ i,j$, where $$\delta_{ij}=\Big\{\begin{array}{cc}
   		1  & \mbox{if } \; i=j,  \\
   		0 & \mbox{other wise. }
   	\end{array}$$
   \end{prop}

   \begin{proof}
   	Let $\{f_j\}_{j=1}^n$ be the set of linear functionals on $V$ such that $f_j(v_i)=\delta_{ij}$ for all $1\leq i,j\leq n$.  Suppose $\lambda_i\in\mbb{C}$ such that $\sum_{i=1}^n\lambda_iv_i=0$. Fix $1\leq i_0\leq n$, then $f_{i_0}(\sum_{i}\lambda_iv_i)=0$. This implies that $\lambda_{i_0}=0$. Since $i_0$ is arbitrary, we conclude that $\lambda_i=0$ for all $i$. This shows that the set $\{v_i\}_{i=1}^n$ is linearly independent. Conversely, if $\{v_i\}_{i=1}^n$ is linearly independent, then extend $\{v_i\}_{i=1}^n$ to a basis $\msc{B}$ of $V$. Now, for each $1\leq j\leq n$, define the linear functional on the basis elements as $f_j(v_j)=1$ and zero for all other basis elements. This will give the set of linear functionals $\{f_j\}_{j=1}^n$ such that $f_j(v_i)=\delta_{ij}$ for all $1\leq i,j\leq n$.
   \end{proof}

   \begin{lem}\label{lem-tensor-LI}
   	Let $V, W$ be two finite-dimensional vector spaces. If $\{v_i\}_{i=1}^{n_1}$ and $\{w_j\}_{j=1}^{n_2}$ are linearly independent subsets of $V$ and $W$ respectively, then $\{v_i\otimes w_j: 1\leq i\leq n_1, 1\leq j\leq n_2\}$ is a linearly independent subset of the vector space $V\otimes W$.
   \end{lem}

   \begin{proof}
   	Let $\{v_i\}_{i=1}^{n_1}$ and $\{w_j\}_{j=1}^{n_2}$ be linearly independent subsets of $V$ and $W$ respectively. Then, by Proposition \ref{prop-LI-functionals}, there exists the set of linear functionals $\{f_k\}_{k=1}^{n_1}$ and $\{g_l\}_{l=1}^{n_2}$ such that 
   	\begin{align*}
   		f_k(v_i)&=\delta_{ik} \mbox{ for all } 1\leq i,k\leq n_1,\\
   		g_l(w_j)&=\delta_{jl} \mbox{ for all } 1\leq j,l\leq n_2.
   	\end{align*}
   	Let $\lambda_{ij}\in\mbb{C}$ such that $\sum_{i,j}\lambda_{ij}v_i\otimes w_j=0$. Fix $1\leq i_0\leq n_1, 1\leq j_0\leq n_2$, then 
   	\begin{align*}
   		0=f_{i_0}\otimes g_{j_0}\left(\sum_{i,j}\lambda_{ij}v_i\otimes w_j\right)=\lambda_{i_0j_0}.
   	\end{align*}
   	Since $i_0,j_0$ are arbitrary, we conclude that $\lambda_{ij}=0$ for all  $1\leq i\leq n_1, 1\leq j\leq n_2$. This shows that the set $\{v_i\otimes w_j: 1\leq i\leq n_1, 1\leq j\leq n_2\}$ is a linearly independent subset of the vector space $V\otimes W$.
   \end{proof}

   \begin{lem}\label{lem-product-Choirank}
   	Let $\Phi:\M{d_1}\to\M{d_2}$ and $\Psi:\M{d_3}\to\M{d_4}$ be two CP maps. Then $\Phi\otimes \Psi$ is a CP map with $CR(\Phi\otimes\Psi)=CR(\Phi)CR(\Psi)$.
   \end{lem}

   \begin{proof}
   	Let $\Phi$ and $\Psi$ be two CP maps with $CR(\Phi)=r_1$ and $CR(\Psi)=r_2$. This means that, there exist linearly independent subsets $\{V_i\}_{i=1}^{r_1}\subseteq\M{d_1\times d_2}$ and $\{W_j\}_{j=1}^{r_2}\subseteq\M{d_3\times d_4}$ such that
   	
   	\begin{align*}
   		\Phi=\sum_{i=1}^{r_1}\mathrm{Ad}_{V_i}\quad\text{and}\quad
   		\Psi=\sum_{j=1}^{r_2}\mathrm{Ad}_{W_j}.
   	\end{align*}
   	
   	Now, for any $A\in\M{d_1}$ and $B\in\M{d_3}$, we have
   	
   	\begin{align*}
   		\Phi\otimes\Psi(A\otimes B) &= \sum_{i=1}^{r_1}\mathrm{Ad}_{V_i}(A)\otimes\sum_{j=1}^{r_2}\mathrm{Ad}_{W_j}(B) \\
   		&= \left(\sum_{i=1}^{r_1}V_i^*AV_i\right)\otimes\left(\sum_{j=1}^{r_2}W_j^*BW_j\right) \\
   		&= \sum_{i,j}(V_i\otimes W_j)^*(A\otimes B)(V_i\otimes W_j).
   	\end{align*}
   	
   	Since every element of $\M{d_1}\otimes\M{d_3}$ is a finite sum of simple tensors, we conclude that $\Phi\otimes\Psi(Z) = \sum_{i,j}(V_i\otimes W_j)^*Z(V_i\otimes W_j)$ for all $Z\in\M{d_1}\otimes\M{d_3}$. Therefore, by Corollary \ref{cor-choi-kraus-rep}, $\Phi\otimes\Psi$ is a CP map. Moreover, from Lemma \ref{lem-tensor-LI}, we have that $\{V_i\otimes W_j: 1\leq i\leq r_1, 1\leq j\leq r_2\}$ is linearly independent. Hence $CR(\Phi\otimes\Psi)=r_1r_2=CR(\Phi)CR(\Psi)$.
   \end{proof}

   \begin{thm}
   	Let $A\in\M{d_2},B\in\M{d_4}, \Phi\in\msc{CP}[\M{d_1},\M{d_2}; A]$ and  $\Psi\in\msc{CP}[\M{d_3},\M{d_4}; B]$. If $\Phi$ and $\Psi$ are extreme points of the set $\msc{CP}[\M{d_1},\M{d_2}; A]$ and $\msc{CP}[\M{d_3},\M{d_4};B]$ respectively, then $\Phi\otimes \Psi$ is also an extreme point of the set $\msc{CP}[\M{d_1}\otimes\M{d_3},\M{d_2}\otimes\M{d_4}; A\otimes B]$.
   \end{thm}

   \begin{proof}
   	Let $\Phi \in \msc{CP}[\M{d_1}, \M{d_2}; A]$ and $\Psi \in \msc{CP}[\M{d_3}, \M{d_4}; B]$ be extreme points of their respective sets $\msc{CP}[\M{d_1}, \M{d_2}; A]$ and $\msc{CP}[\M{d_3}, \M{d_4}; B]$. By Theorem \ref{choi}, there exist Kraus operators $\{V_j\}_{j=1}^{r_1}\subseteq\M{d_1 \times d_2}$ and $\{W_l\}_{l=1}^{r_2} \subseteq\M{d_3 \times d_4}$ such that $\Phi = \sum_{j=1}^{r_1} \mathrm{Ad}_{V_j}$ and $\Psi = \sum_{l=1}^{r_2} \mathrm{Ad}_{W_l}$, and the sets $\{V_i^*V_j\}_{i,j=1}^{r_1}$ and $\{W_k^*W_l\}_{k,l=1}^{r_2}$ are linearly independent. Given that $\Phi(I_{d_1}) = A$ and $\Psi(I_{d_3}) = B$, therefore $\Phi \otimes \Psi(I_{d_1} \otimes I_{d_3}) = A \otimes B$. Furthermore, from Lemma \ref{lem-product-Choirank}, we have $\Phi \otimes \Psi$ is CP and  hence $\Phi \otimes \Psi \in \msc{CP}[\M{d_1} \otimes \M{d_3}, \M{d_2} \otimes \M{d_4}; A \otimes B]$ with a Kraus decompositon $\sum_{j,l=1}^{r_1,r_2} \mathrm{Ad}_{V_j \otimes W_l}$. Since the sets $\{V_i^*V_j\}_{i,j=1}^{r_1}$ and $\{W_k^*W_l\}_{k,l=1}^{r_2}$ are linearly independent, from Lemma \ref{lem-tensor-LI}, we also have that the set $\{(V_i \otimes W_k)^*(V_j \otimes W_l)\}_{i,j,k,l}$ is linearly independent. Thus, by applying Theorem \ref{choi} again, we conclude that $\Phi \otimes \Psi$ is an extreme point of the set $\msc{CP}[\M{d_1} \otimes \M{d_3}, \M{d_2} \otimes \M{d_4}; A \otimes B]$.
   \end{proof}

   However, given $A_i\in\M{d_i}$ for $i=1,2$ and $B_j\in\M{d_j}$ for $j=3,4$, the tensor product of extreme points of the sets $\msc{CP}[\M{d_1},\M{d_2}; A_1,A_2]$ and $\msc{CP}[\M{d_3},\M{d_4}; B_1,B_2]$ is not necessarily extreme in the set $\msc{CP}[\M{d_1}\otimes\M{d_3},\M{d_2}\otimes\M{d_4}; A_1\otimes B_1,A_2\otimes B_2]$. For example, let $\Phi\in\msc{CP}[\M{3},\M{3}; I_3,I_3]$ and $\Psi\in\msc{CP}[\M{4},\M{4}; I_4,I_4]$ are extreme points of their respective sets $\msc{CP}[\M{3},\M{3}; I_3,I_3]$ and $\msc{CP}[\M{4},\M{4}; I_4,I_4]$ with $CR(\Phi)=4$ and $CR(\Psi)=5$, as given in Theorem \ref{thm-extreme-rank-4} and \ref{thm-extreme-rank-5}. By Lemma \ref{lem-product-Choirank}, $\Phi\otimes\Psi\in \msc{CP}[\M{12},\M{12}; I_{12},I_{12}]$ with $CR(\Phi\otimes\Psi)=20$, . But $\Phi\otimes\Psi$ is not an extreme point of $\msc{CP}[\M{12},\M{12}; I_{12},I_{12}]$ because if so, by the upper bound on the Choi-rank given by K.R. Parthasarathy, we should have $CR(\Phi\otimes\Psi)\leq\lfloor{\sqrt{{12}^2+{12}^2-1}}\rfloor=16$.

   The following theorem provides a sufficient condition for the tensor product of extreme points to be extreme. This is the main result of this section. We will use this result in the next section to obtain various classes of extremal marginal states that attain the upper bound given by K. R. Parthasarathy.

   \begin{thm}\label{thm-tensor product-extreme}
   	Let $\Phi\in\msc{CP}[\M{d},\M{d}; A_1, A_2]\subseteq \msc{CP}[\M{d},\M{d}; A_2]$, which has a minimal Kraus decomposition consisting of Hermitian Kraus operators. Let $\Psi\in\msc{CP}[\M{d_1},\M{d_2}; B_1, B_2]$. If $\Phi$ and $\Psi$ are both extreme points of the sets $\msc{CP}[\M{d},\M{d}; A_2]$ and $\msc{CP}[\M{d_1},\M{d_2}; B_1, B_2]$, respectively, then the tensor product $\Phi\otimes \Psi$ is also an extreme point of the set $\msc{CP}[\M{d}\otimes\M{d_1},\M{d}\otimes\M{d_2}; A_1\otimes B_1, A_2\otimes B_2]$.
   \end{thm}

   \begin{proof} 
   	Let $\{V_j\}_{j=1}^{r_1}\subseteq\M{d}$ be a Hermitian family that forms a minimal Kraus decomposition of $\Phi$. That is,  $\Phi=\sum_{j=1}^{r_1}\mathrm{Ad}_{V_j}$ with $V_j=V_j^*$ for all $1\leq j \leq r_1$ and $CR(\Phi)=r_1$. Since $\Psi\in\msc{CP}[\M{d_1},\M{d_2}; B_1,B_2]$ is any extreme point of $\msc{CP}[\M{d_1},\M{d_2};B_1, B_2]$, by Theorem \ref{thm-extr-condition-for doubly}, there exists $\{W_l\}_{l=1}^{r_2}\subseteq\M{d_1\times d_2}$ such that $\Psi=\sum_{l=1}^{r_2}\mathrm{Ad}_{W_l}$, and the pair of sets $\{W_k^*W_l\}_{k,l=1}^{r_2}$ and $\{W_lW_k^*\}_{k,l=1}^{r_2}$ is bi-linearly independent. Since $\Phi\in\msc{CP}[\M{d},\M{d}; A_1,A_2]$ and $\Psi\in\msc{CP}[\M{d_1},\M{d_2}; B_1,B_2]$, by Lemma \ref{lem-product-Choirank}, we have $\Phi\otimes\Psi\in\msc{CP}[\M{d}\otimes\M{d_1},\M{d}\otimes\M{d_2}; A_1\otimes B_1,A_2\otimes B_2]$ and $\Phi\otimes\Psi=\sum_{j,l=1}^{r_1,r_2}\mathrm{Ad}_{V_j\otimes W_l}$. To show that $\Phi\otimes\Psi$ is an extreme point of the set $\msc{CP}[\M{d}\otimes\M{d_1},\M{d}\otimes\M{d_2}; A_1\otimes B_1, A_2\otimes B_2]$, it is enough to prove that the pair of sets $\{(V_i\otimes W_k)^*(V_j\otimes W_l):1\leq i,j\leq r_1, 1\leq k,l\leq r_2\}$ and $\{(V_j\otimes W_l)(V_i\otimes W_k)^*:1\leq i,j\leq r_1, 1\leq k,l\leq r_2\}$ is bi-linearly independent. Let $a_{ijkl}\in\mbb{C}$ such that
   	
   	\begin{align}
   		0&=\sum_{i,j,k,l}a_{ijkl}(V_i\otimes W_k)^*(V_j\otimes W_l)=\sum_{i,j,k,l}a_{ijkl}(V_iV_j\otimes W_k^*W_l) \label{eq-bi-Li-1},\\
   		0&=\sum_{i,j,k,l}a_{ijkl}(V_j\otimes W_l)(V_i\otimes W_k)^*=\sum_{i,j,k,l}a_{ijkl}(V_jV_i\otimes W_lW_k^*)\label{eq-bi-Li-2}.
   	\end{align}
   	
   	Since $\Phi$ is an extreme point of $\msc{CP}[\M{d},\M{d};A_2]$, by Remark \ref{remark}, the set $\{V_i^*V_j\}_{i,j=1}^{r_1}=\{V_iV_j\}_{i,j=1}^{r_1}$ is linearly independent. By Lemma \ref{prop-LI-functionals}, there exists a family of functionals $\{f_{ij}\}_{i,j=1}^{r_1}$ on $\M{d}$ such that $f_{ij}(V_mV_n)=\delta_{im}\delta_{jn}$ for all $1\leq i,j,m,n\leq r_1$. Fix $1\leq i_0,j_0\leq r_1$. Applying $f_{i_0j_0}\otimes \id$ and $f_{j_0i_0}\otimes\id$ to Equations \ref{eq-bi-Li-1} and \ref{eq-bi-Li-2} respectively, we get
   	
   	\begin{align*}
   		\sum_{k,l}a_{i_0j_0kl}W_k^*W_l=0,\\
   		\sum_{k,l}a_{i_0j_0kl}W_lW_k^*=0.
   	\end{align*}
   	
   	As the pair of the sets $\{W_k^*W_l\}_{k,l=1}^{r_2}$ and $\{W_lW_k^*\}_{k,l=1}^{r_2}$ is bi-linearly independent, we get $a_{i_0j_0kl}=0$ for all $1\leq k,l\leq r_2$. Since $i_0,j_0$ are arbitrary, we conclude that $a_{ijkl}=0$ for all $1\leq i,j\leq r_1$ and $1\leq k,l\leq r_2$. This completes the proof.
   \end{proof}

   \begin{rmk}
   	Let $\Phi\in\msc{CP}(\M{d},\M{d}; I_d, I_d)$. If $\Phi=\sum_{j=1}^{n}\mathrm{Ad}_{V_j}$ is an extreme point of $\msc{CP}(\M{d},\M{d}; I_d)$, then $\Phi^*=\sum_{j=1}^{n}\mathrm{Ad}_{V_j^*}$ is not necessarily an extreme point of $\msc{CP}(\M{d},\M{d}; I_d)$. That is, if $\{V_i^*V_j\}_{i,j=1}^n$ is linearly independent with $\sum_{i,j=1}^nV_i^*V_j=\sum_{i,j}^nV_jV_i^*=I_d$, then $\{V_jV_i^*\}_{i,j=1}^n$ is not necessarily linearly independent. For example, take $d=4$ and 
    \begin{align*}
        V_j=\frac{3}{4\sqrt{11}}{S^j}^*\begin{bmatrix}
			-\frac{13}{3} & 0\\
			0 & W\\
		\end{bmatrix}S^j
    \end{align*}
    for $1\leq j\leq 4$, where
    \begin{align*}
        &S=\sum_{k=1}^3\ranko{e_k}{e_{k+1}}+\ranko{e_4}{e_1},\\
        &W=\frac{1}{21}\begin{bmatrix}
			8 & -11 & 16\\
			-19 & -8 & 4\\
            -4 & 16 & 13\\
		\end{bmatrix}.
    \end{align*}
    Check that $\{V_i^*V_j\}_{i,j=1}^4$ is linearly independent with $\sum_{i,j=1}^4V_i^*V_j=\sum_{i,j}^4V_jV_i^*=I_4$ but $\{V_jV_i^*\}_{i,j=1}^4$ is linearly dependent. The reader may refer to Appendix B of \cite{HMR} for more details.
   \end{rmk}

\section{Sharpness of the upper bound for the case $d_1=d_2$}\label{sharpness for d_1=d_2}

Given $d\geq 3$, Ohno \cite{Ohno10} proved that there exists an extreme point in the set $\msc{CP}(\M{d},\M{d}; I_{d}, I_{d})$ with Choi-rank $d$. Furthermore, he established that the upper bound on the rank given by K. R. Parthasarathy is indeed sharp for $d_1 = d_2 = 3$ and $d_1 = d_2 = 4$. That is, there exist extreme points of the sets $\msc{CP}(\M{3},\M{3}; I_{3}, I_{3})$ and $\msc{CP}(\M{4},\M{4}; I_{4}, I_{4})$ with Choi-ranks $4$ and $5$ respectively. In this section, we will examine the sharpness of the upper bound given by K. R. Parthasarathy for the cases when $d_1=d_2\geq 5$.

\begin{center}
	 \item\subsection{$d_1=d_2=5$}
\end{center}

In this subsection, we provide examples of extreme points belonging to the set $\msc{CP}(\M{5},\M{5}; I_{5}, I_{5})$, which have Choi-ranks $6$ and $7$. Notably, as $\lfloor\sqrt{5^2+5^2-1}\rfloor=7$, we can affirm that the upper bound suggested by K. R. Parthasarathy is indeed sharp for the case $d_1=d_2=5$.

\begin{lem}\label{lem-CR-6}
	Let $\{W_j:1\leq j\leq 6\}\subseteq\M{5}$ be the set of matrices given by 
	\small
	\begin{align}\label{eq-CR-6}
		\begin{cases}
			W_1&=E_{13}+E_{32}=[0,e_3,e_1,0,0],\\
			W_2&=E_{24}+E_{43}=[0,0,e_4,e_2,0],\\
			W_3&=\sqrt{2}E_{35}+E_{54}=[0,0,0,e_5,\sqrt{2}e_3],\\
			W_4&=E_{14}+E_{42}=[0,e_4,0,e_1,0],\\
			W_5&=E_{15}+E_{41}+E_{53}=[e_4,0,e_5,0,e_1],\\
			W_6&=\sqrt{2}E_{21}+E_{52}=[\sqrt{2}e_2,e_5,0,0,0].
		\end{cases}
	\end{align}
	\normalsize
	Then the pair of sets $\{W_i^*W_j:1\leq i,j\leq 6\}$ and $\{W_jW_i^*:1\leq i,j\leq 6\}$ is bi-linearly independent.
\end{lem}

\begin{proof}
	See the Lemma \ref{lem-CR-6-appendix} in Appendix.
\end{proof}

 \begin{thm}
	There exists an extreme point of the set $\msc{CP}(\M{5},\M{5}; I_{5}, I_{5})$ with Choi-rank $6$.
\end{thm}

 \begin{proof}
	Let $\Phi$ be the map defined on $\M{5}$ by
	\begin{align*}
		\Phi(X)
		&=\frac{1}{3}\sum_{j=1}^{6}W_j^*XW_j\\
		&=\frac{1}{3}\Matrix{2x_{22}+x_{44}&\sqrt{2}x_{25}&x_{45}&0&x_{41}\\
			\sqrt{2}x_{52}&x_{33}+x_{44}+x_{55}&x_{31}&x_{41}&0\\
			x_{54}&x_{13}&x_{11}+x_{44}+x_{55}&x_{42}&x_{51}\\
			0&x_{14}&x_{24}&x_{11}+x_{22}+x_{55}&\sqrt{2}x_{53}\\
			x_{14}&0&x_{15}&\sqrt{2}x_{35}&x_{11}+2x_{33}}
	\end{align*}
	for all $X=[x_{ij}]\in\M{5}$, where $W_j$'s are given by \eqref{eq-CR-6}. Clearly, $\Phi\in\msc{CP}(\M{5},\M{5}; I_{5}, I_{5})$. To verify that $\Phi$ is an extreme point of the set $\msc{CP}(\M{5},\M{5}; I_{5}, I_{5})$, it suffices to show that the pair of sets $\{W_i^*W_j:1\leq i,j\leq 6\}$ and $\{W_jW_i^*:1\leq i,j\leq 6\}$ is bi-linearly independent. This assertion indeed follows from the Lemma \ref{lem-CR-6}. Consequently, the map $\Phi$ is an extreme point with $CR(\Phi)=6$.
\end{proof}

 \begin{lem}\label{lem-CR-7}
	Let $\{W_j:1\leq j\leq 7\}\subseteq\M{5}$ be the set of matrices given by
	\begin{align}\label{eq-CR-7}
		\begin{cases}
			W_1&=\sqrt{2}E_{13}+E_{32}+E_{54}=[0,e_3,\sqrt{2}e_1,e_5,0],\\
			W_2&=E_{24}+E_{43}=[0,0,e_4,e_2,0],\\
			W_3&=\sqrt{2}E_{35}+\sqrt{3}E_{54}=[0,0,0,\sqrt{3}e_5,\sqrt{2}e_3],\\
			W_4&=E_{14}+E_{42}=[0,e_4,0,e_1,0]\\
			W_5&=E_{15}+2E_{41}+E_{53}=[2e_4,0,e_5,0,e_1],\\
			W_6&=\sqrt{2}E_{21}+E_{52}=[\sqrt{2}e_2,e_5,0,0,0],\\
			W_7&=\sqrt{2}E_{13}+\sqrt{3}iE_{32}+\sqrt{3}E_{25}=[0,\sqrt{3}ie_3,\sqrt{2}e_1,0,\sqrt{3}e_2].
		\end{cases}
	\end{align}
	Then the pair of sets $\{W_i^*W_j:1\leq i,j\leq 7\}$ and $\{W_jW_i^*: 1\leq i,j\leq 7\}$ is bi-linearly independent.
\end{lem}

\begin{proof}
	See the Lemma \ref{lem-CR-7-appendix} in Appendix.
\end{proof}

\begin{thm}\label{thm-CR-7}
	There exists an extreme point of the set $\msc{CP}(\M{5},\M{5}; I_{5}, I_{5})$ with Choi-rank $7$.
\end{thm}

\begin{proof}
	Define the linear map $\Phi$ on $\M{5}$ by
	\small
	\begin{align*}
		\Phi(X)
		&=\frac{1}{6}\sum_{j=1}^{7}W_j^*XW_j\\
		&=\frac{1}{6}\Matrix{2x_{22}+4x_{44}&\sqrt{2}x_{25}&2x_{45}&0&2x_{41}\\
			\sqrt{2}x_{52}&4x_{33}+x_{44}+x_{55}&(\sqrt{2}-\sqrt{6}i)x_{31}&x_{41}+x_{35}&-3ix_{32}\\
			2x_{54}&(\sqrt{2}+\sqrt{6}i)x_{13}&4x_{11}+x_{44}+x_{55}&x_{42}+\sqrt{2}x_{15}&x_{51}+\sqrt{6}x_{12}\\
			0&x_{53}+x_{14}&x_{24}+\sqrt{2}x_{51}&x_{11}+x_{22}+4x_{55}&\sqrt{6}x_{53}\\
			2x_{14}&3ix_{23}&x_{15}+\sqrt{6}x_{21}&\sqrt{6}x_{35}&x_{11}+3x_{22}+2x_{33}}
	\end{align*}
	\normalsize
	for all $X=[x_{ij}]\in\M{5}$, where $W_j$'s are given by \eqref{eq-CR-7}. Clearly, $\Phi\in\msc{CP}(\M{5},\M{5}; I_{5}, I_{5})$. Now, from Lemma \ref{lem-CR-7}, it follows that the pair of the sets $\{W_i^*W_j:1\leq i,j\leq 7\}$ and $\{W_jW_i^*:1\leq i,j\leq 7\}$ is bi-linearly independent and hence from Theorem \ref{thm-extr-condition-for doubly}, we conclude that $\Phi$ is extreme point of the set $\msc{CP}(\M{5},\M{5}; I_{5}, I_{5})$ with $CR(\Phi)=7$.
\end{proof}

\begin{cor}
	$\msc{MR}(\frac{I_{5}}{5},\frac{I_{5}}{5})=7=\lfloor{\sqrt{5^2+5^2-1}}\rfloor$.
\end{cor}

\begin{proof}
	Take $\Phi_1=\frac{\Phi}{5}$, where $\Phi$ is the same as considered in the proof of the Theorem \ref{thm-CR-7}. Then $\Phi_1$ is an extreme point within the set $\msc{CP}(\M{5},\M{5}; \frac{I_{5}}{5},\frac{I_{5}}{5})$ with $CR(\Phi_1)=CR(\Phi)=7$. Hence we conclude that $\msc{MR}(\frac{I_{5}}{5},\frac{I_{5}}{5})=7$.
\end{proof}

\begin{center}
	\item \subsection{ $(d_1,d_2)=(9,9), (12,12), (5k,5k), 3\leq k\leq 14 $}
\end{center}

In this subsection, we study that the upper bound given by K. R. Parthasarathy is indeed sharp for $(d_1,d_2)=(9,9), (12,12), (5k,5k) \text{ where } 3\leq k\leq 14 $. That is, there exist extreme points of the sets $\msc{CP}(\M{9},\M{9}; I_{9}, I_{9}), \msc{CP}(\M{12},\M{12}; I_{12}, I_{12})$ and $\msc{CP}(\M{5k},\M{5k}; I_{5k}, I_{5k})$ with Choi-ranks $12, 16$ and $7k$ respectively, where $3\leq k\leq 14$. To construct these examples, we will utilize the examples given by Ohno \cite{Ohno10} and the example presented in the previous subsection.

\begin{thm}\label{thm-krpbound-M9}
	There exists an extreme point of the set  $\msc{CP}(\M{9},\M{9}; I_{9},I_{9})$ with Choi-rank $12$. 
\end{thm}

\begin{proof}
	Let $\Phi_1$ and $\Phi_2$ be two CP maps on $\M{3}$, defined by the Equations \ref{eq-extreme-rank-lowerbound-2} and \ref{eq-extreme-rank-4-2}, respectively. That is,
	\begin{align*}
		\Phi_1(X) &= \sum_{i=1}^3 V_i^* X V_i = \frac{1}{2}\begin{bmatrix} 
			x_{22} + x_{33} & x_{21} & x_{31} \\
			x_{12} & x_{11} + x_{22} & x_{23} \\
			x_{13} & x_{32} & x_{11} + x_{33}
		\end{bmatrix}, \\
		\Phi_2(X) &= \frac{1}{4}\sum_{j=1}^4 W_j^* X W_j = \frac{1}{4}\begin{bmatrix}
			x_{11} + 2x_{22} + x_{33} & \sqrt{6}x_{23} & \sqrt{2}x_{31} \\
			\sqrt{6}x_{32} & 3x_{33} + x_{11} & \sqrt{2}x_{12} \\
			\sqrt{2}x_{13} & \sqrt{2}x_{21} & 2x_{22} + 2x_{11}
		\end{bmatrix},
	\end{align*}
	where the matrices $V_i$ and $W_j$ are given by \eqref{eq-extreme-rank-lowerbound} and \eqref{eq-extreme-rank-4} respectively, for  $1\leq i\leq 3$ and $1\leq j\leq 4$. Clearly, $\Phi_1,\Phi_2\in\msc{CP}(\M{3},\M{3}; I_{3},I_{3})$. By Theorem \ref{thm-extrem-rank-lowerbound}, $\Phi_1$ is an extreme point of the set $\msc{CP}(\M{3},\M{3}; I_{3})$ with $CR(\Phi_1)=3$, which has a minimal Kraus decomposition consisting of Hermitian Kraus operators $\{V_i\}_{i=1}^3$. Also, by Theorem \ref{thm-extreme-rank-4}, $\Phi_2$ is an extreme point of the set $\msc{CP}(\M{3},\M{3}; I_{3}, I_{3})$ with $CR(\Phi_2)=4$. Take $\Phi = \Phi_1 \otimes \Phi_2$. By Lemma \ref{lem-product-Choirank}, it follows that $\Phi\in\msc{CP}(\M{9},\M{9}; I_{9}, I_{9})$ with $CR(\Phi)=12$. Finally, by Theorem \ref{thm-tensor product-extreme}, we conclude that $\Phi$ is an extreme point of the set $\msc{CP}(\M{9},\M{9}; I_{9}, I_{9})$ with $CR(\Phi)=12$.
\end{proof}

\begin{cor}
	$\msc{MR}(\frac{I_{9}}{9},\frac{I_{9}}{9})= 12=\lfloor{\sqrt{9^2+9^2-1}}\rfloor$.
\end{cor}

\begin{proof}
	Take $\Phi_1=\frac{\Phi}{9}$, where $\Phi$ is the same as considered in the proof of the Theorem \ref{thm-krpbound-M9}. Then $\Phi_1$ is an extreme point within the set $\msc{CP}(\M{9},\M{9}; \frac{I_{9}}{9},\frac{I_{9}}{9})$ with $CR(\Phi_1)=CR(\Phi)=12$. Hence we conclude that $\msc{MR}(\frac{I_{9}}{9},\frac{I_{9}}{9})=12$.
\end{proof}

\begin{thm}\label{thm-CR-16 in 12}
	There exists an extreme point of the set $\msc{CP}(\M{12},\M{12};I_{12},I_{12})$ with Choi-rank $16$.
\end{thm}

\begin{proof}
	Let $\Phi_1$ be a CP map on $\M{4}$ given by 
	\begin{align*}
		\Phi_1(X) &=\sum_{i=1}^4V_i^*XV_i= \frac{1}{3}\begin{bmatrix}
			x_{22}+x_{33}+x_{44} & x_{21} & x_{31} & x_{41}\\
			x_{12} & x_{11}+2x_{22} & 2x_{23} & 2x_{24}\\
			x_{13} & 2x_{32} & x_{11}+2x_{33} & 2x_{34}\\
			x_{14} & 2x_{42} & 2x_{43} & x_{11}+2x_{44}
		\end{bmatrix}
	\end{align*}
	for all $X=[x_{ij}]\in\M{4}$, where $V_i$'s are given by \eqref{eq-extreme-rank-lowerbound} for $1\leq i\leq 4$.
	
	Similarly, let $\Phi_2$ be the CP map on $\M{3}$ given by 
	\begin{align*}
		\Phi_2(X) &= \frac{1}{4}\sum_{j=1}^4 W_j^* X W_j = \frac{1}{4}\begin{bmatrix}
			x_{11} + 2x_{22} + x_{33} & \sqrt{6}x_{23} & \sqrt{2}x_{31} \\
			\sqrt{6}x_{32} & 3x_{33} + x_{11} & \sqrt{2}x_{12} \\
			\sqrt{2}x_{13} & \sqrt{2}x_{21} & 2x_{22} + 2x_{11}
		\end{bmatrix}
	\end{align*}
	for all $X=[x_{ij}]\in\M{3}$, where $W_j$'s are given by \eqref{eq-extreme-rank-4} for $1\leq j\leq 4$. Clearly, $\Phi_1\in\msc{CP}(\M{4},\M{4};I_{4},I_{4})$ and $\Phi_2\in\msc{CP}(\M{3},\M{3}; I_{3},I_{3})$. By Theorem \ref{thm-extrem-rank-lowerbound}, $\Phi_1$ is an extreme point of the set $\msc{CP}(\M{4},\M{4}; I_{4})$ with $CR(\Phi_1)=4$, which has a minimal Kraus decomposition consisting of Hermitian Kraus operators $\{V_i\}_{i=1}^3$. Also, by Theorem \ref{thm-extreme-rank-4}, $\Phi_2$ is an extreme point of the set $\msc{CP}(\M{3},\M{3}; I_{3}, I_{3})$ with $CR(\Phi_2)=4$. Take $\Phi = \Phi_1 \otimes \Phi_2$. By Lemma \ref{lem-product-Choirank}, it follows that $\Phi\in\msc{CP}(\M{12},\M{12}; I_{12}, I_{12})$ with $CR(\Phi)=16$. Finally, by Theorem \ref{thm-tensor product-extreme}, we conclude that $\Phi$ is an extreme point of the set $\msc{CP}(\M{12},\M{12}; I_{12}, I_{12})$ with $CR(\Phi)=16$.
\end{proof}

\begin{cor}
	$\msc{MR}(\frac{I_{12}}{12},\frac{I_{12}}{12})= 16=\lfloor{\sqrt{12^2+12^2-1}}\rfloor$.
\end{cor}

\begin{proof}
	Follows from Theorem \ref{thm-CR-16 in 12}.
\end{proof}

\begin{thm}\label{thm-CR-7k in 5k}
	There exists an extreme point of the set $\msc{CP}(\M{5k},\M{5k};I_{5k},I_{5k})$ with Choi-rank $7k$, for all $k\geq 3$.
\end{thm}

\begin{proof}
	Given $k\geq 3$. Let $\Phi_1\in\msc{CP}(\M{k},\M{k};I_{k},I_{k})$ be an extreme point of the set $\msc{CP}(\M{k},\M{k}; I_{k})$ with $CR(\Phi_1)=k$, which has a minimal Kraus decomposition consisting of Hermitian Kraus operators. Suppose $\Phi_2\in\msc{CP}(\M{5},\M{5}; I_{5},I_{5})$ is an extreme point of the set $\msc{CP}(\M{5},\M{5}; I_{5}, I_{5})$ with $CR(\Phi_2)=7$. Such $\Phi_1$ and $\Phi_2$ exist because of Theorem \ref{thm-extrem-rank-lowerbound} and \ref{thm-CR-7} respectively. Take $\Phi = \Phi_1 \otimes \Phi_2$. By Lemma \ref{lem-product-Choirank}, it follows that $\Phi\in\msc{CP}(\M{5k},\M{5k}; I_{5k}, I_{5k})$ with $CR(\Phi)=7k$. Using Theorem \ref{thm-tensor product-extreme}, we conclude that $\Phi$ is an extreme point of the set $\msc{CP}(\M{5k},\M{5k}; I_{5k}, I_{5k})$ with $CR(\Phi)=7k$.
\end{proof}

Note that, if $k\leq 14$ then $(7k)^2\leq 2(5k)^2-1< (7k+1)^2$. Hence $\lfloor{\sqrt{(5k)^2+(5k)^2-1}}\rfloor=7k$ for all $k\leq 14$.

\begin{cor}
	If $3\leq k\leq 14$, then $\msc{MR}(\frac{I_{5k}}{5k},\frac{I_{5k}}{5k})=7k=\lfloor{\sqrt{(5k)^2+(5k)^2-1}}\rfloor$. 
\end{cor}

\begin{proof}
	Follows from the Theorem \ref{thm-CR-7k in 5k}.
\end{proof}

\section{Sharpness of the upper bound for the case \texorpdfstring{$d_1\neq d_2$}{TEXT}}\label{section 5}

In \cite{KS21}, the authors established the sharpness of the upper bound on the rank of extreme points given by K. R. Parthasarathy for the qubit-qutrit system. Specifically, they illustrate an example of an extreme point $\rho\in\msc{C}(\frac{I_{2}}{2},\frac{I_{3}}{3})$ with $\text{rank}(\rho)=3=\lfloor{\sqrt{2^2+3^2-1}}\rfloor$. Equivalently, there exists an extreme point $\Phi\in\msc{CP}(\M{2}, \M{3}; \frac{I_2}{2}, \frac{I_3}{3})$ with $CR(\Phi)=3$. This is now natural to ask whether this holds true in a general context when $d_1\neq d_2$. In this section, we obtain that the upper bound given by K. R. Parthasarathy is indeed sharp for $(d_1,d_2)=(2,d), (3,4)$, where $d\geq 4$. This section also contains an example of an extreme point of $\msc{CP}(\M{d}, \M{d+1}; \frac{I_d}{d}, \frac{I_{d+1}}{d+1})$ with Choi-rank $d+1$ for $d\geq 2$.

\begin{center}
	\item \subsection{$(d_1,d_2)=(2,d), d\geq 4$}
\end{center}

Given $d\geq 4$. In this subsection, we construct an example of an extreme point of the set $\msc{CP}(\M{2},\M{d}; Z, \frac{I_d}{d})$ with Choi-rank $d$, where $Z$ is an invertible state on $\M{2}$ given by
\begin{align}\label{definition of Z}
	 Z=\frac{1}{2}\Matrix{1&\frac{d-3}{d}\\\frac{d-3}{d}&1}.
\end{align}
Since $\lfloor{\sqrt{2^2+d^2-1}}\rfloor=d$, this asserts that the upper bound given by K. R. Parthasarathy is indeed sharp for $(d_1,d_2)=(2,d)$ where $d\geq 4$. We begin with the following lemma.

\begin{lem}\label{lem-CR-d in 2 to d}
	Given $d\geq 4$. Let $\{W_j\}_{j=1}^d\subseteq\M{2\times d}$ be the set of matrices given by
	\begin{align}\label{eq-CR-d in 2 to d}
		\begin{cases}
			W_1&=\ranko{e_1^{(2)}}{e_1^{(d)}}+\ranko{e_2^{(2)}}{e_3^{(d)}}=[e_1^{(2)},0,e_2^{(2)},0,\hdots, 0],\\
			W_2&=\ranko{e_1^{(2)}}{e_2^{(d)}}+\ranko{e_2^{(2)}}{e_1^{(d)}}=[e_2^{(2)},e_1^{(2)},0,0,\hdots, 0],\\
			W_3&=\ranko{e_1^{(2)}}{e_3^{(d)}}+\ranko{e_2^{(2)}}{e_2^{(d)}}=[0,e_2^{(2)},e_1^{(2)},0,\hdots, 0],\\
			W_4&=\ranko{e_1^{(2)}}{e_4^{(d)}}+\ranko{e_2^{(2)}}{e_4^{(d)}}=[0,0,0,e_1^{(2)}+e_2^{(2)},\hdots, 0],\\
			\vdots\\
			W_d&=\ranko{e_1^{(2)}}{e_d^{(d)}}+\ranko{e_2^{(2)}}{e_d^{(d)}}=[0,0,0,0,\hdots, e_1^{(2)}+e_2^{(2)}].
		\end{cases}
	\end{align}
	Then the pair of sets $\{W_i^*W_j:1\leq i,j\leq d\}\subseteq\M{d}$ and $\{W_jW_i^*:1\leq i,j\leq d\}\subseteq\M{2}$ is bi-linearly independent.
\end{lem}

\begin{proof}
	Let $4\leq r, k\leq d$. We compute

	\footnotesize 
	 	\begin{equation*}
				\begin{aligned}[c]
					W_1^*W_1&=E_{11}+E_{33},\\
					W_1^*W_2&=E_{12}+E_{31},\\
					W_1^*W_3&=E_{13}+E_{32},\\
					W_1^*W_r&=E_{1r}+E_{3r},\\\\
					W_1W_1^*&=F_{11}+F_{22},\\
					W_2W_1^*&=F_{21},\\
					W_3W_1^*&=F_{12},\\
					W_rW_1^*&=0,\\
                \end{aligned}
				\qquad 
				\begin{aligned}[c]
					W_2^*W_1&=E_{21}+E_{13},\\
					W_2^*W_2&=E_{22}+E_{11},\\
					W_2^*W_3&=E_{23}+E_{12},\\
					W_2^*W_r&=E_{2r}+E_{1r},\\\\
					W_1W_2^*&=F_{12},\\
					W_2W_2^*&=F_{11}+F_{22},\\
					W_3W_2^*&=F_{21},\\
					W_rW_2^*&=0,\\
                \end{aligned}
				\qquad
				\begin{aligned}[c]
					W_3^*W_1&=E_{31}+E_{23},\\
					W_3^*W_2&=E_{32}+E_{21},\\
					W_3^*W_3&=E_{33}+E_{22},\\
					W_3^*W_r&=E_{3r}+E_{2r},\\\\
					W_1W_3^*&=F_{21},\\
					W_2W_3^*&=F_{12},\\
					W_3W_3^*&=F_{11}+F_{22},\\
					W_rW_3^*&=0,\\
				\end{aligned}
				\qquad 
				\begin{aligned}[c]
					W_k^*W_1&=E_{k1}+E_{k3},\\
					W_k^*W_2&=E_{k2}+E_{k1},\\
					W_k^*W_3&=E_{k3}+E_{k2},\\
					W_k^*W_r&=2E_{kr},\\\\
					W_1W_k^*&=0,\\
					W_2W_k^*&=0,\\
					W_3W_k^*&=0,\\
					W_rW_k^*&=\delta_{rk}(F_{11}+F_{12}+F_{21}+F_{22})\\
				\end{aligned}
	\end{equation*}
	\normalsize
	 where $F_{ij}:=\ranko{e_i^{(2)}}{e_j^{(2)}}$ for all $1\leq i,j\leq 2$. Let $A=[a_{ij}]\in\M{d}$ such that
	 \begin{align*}
	 	\sum_{i,j=1}^da_{ij}W_i^*W_j=0=\sum_{i,j=1}^da_{ij}W_jW_i^*.
	 \end{align*}

	 A simple calculation yields
	 \begin{align*}
	 	&\Matrix{a_{11}+a_{22}&a_{12}+a_{23}&a_{13}+a_{21}&a_{14}+a_{24}&\cdots&a_{1t}+a_{2t}&\cdots&a_{1d}+a_{2d}\\
	 		a_{21}+a_{32}&a_{22}+a_{33}&a_{23}+a_{31}&a_{24}+a_{34}&\cdots&a_{2t}+a_{3t}&\cdots&a_{2d}+a_{3d}\\
	 		a_{12}+a_{31}&a_{13}+a_{32}&a_{11}+a_{33}&a_{14}+a_{34}&\cdots&a_{1t}+a_{3t}&\cdots&a_{1d}+a_{3d}\\
	 		a_{41}+a_{42}&a_{42}+a_{43}&a_{41}+a_{43}&2a_{44}&\cdots&2a_{4t}&\cdots&2a_{4d}\\
	 		\vdots&\vdots&\vdots&\vdots&\cdots&\vdots&\cdots&\vdots\\
	 		a_{s1}+a_{s2}&a_{s2}+a_{s3}&a_{s1}+a_{s3}&2a_{s4}&\cdots&2a_{st}&\cdots&2a_{sd}\\
	 		\vdots&\vdots&\vdots&\vdots&\cdots&\vdots&\cdots&\vdots\\
	 		a_{d1}+a_{d2}&a_{d2}+a_{d3}&a_{d1}+a_{d3}&2a_{d4}&\cdots&2a_{dt}&\cdots&2a_{dd}}=0
	 	\end{align*}
	 	 and 
	 	 \begin{align*}
	 	{\Matrix{\tr(A)&a_{13}+a_{21}+a_{32}+\sum\limits_{j=4}^da_{jj}\\
	 			a_{12}+a_{23}+a_{31}+\sum\limits_{j=4}^da_{jj}&\tr(A)}=0}.
	 \end{align*}

	 From the above two matrices, we deduce that $a_{ij}=0$ for all $1\leq i,j\leq d$. This shows that the pair of sets $\{W_i^*W_j:1\leq i,j\leq d\}\subseteq\M{d}$ and $\{W_jW_i^*:1\leq i,j\leq d\}\subseteq\M{2}$ is bi-linearly independent.
\end{proof}

\begin{thm}\label{thm- CR-d- 2 to d}
	Let $d\geq 4$. There exists an extreme point of the set $\msc{CP}(\M{2},\M{d}; Z, \frac{I_{d}}{d})$ with Choi-rank $d$, where $Z$ is given in \eqref{definition of Z}.
\end{thm}

\begin{proof}
	Let $\Phi:\M{2}\to\M{d}$ be a CP map defined by
	\begin{align*}
		\Phi(X)
		&=\frac{1}{2d}\sum_{j=1}^{d}W_j^*XW_j\\
		&=\frac{1}{2d}\Matrix{\text{tr}(X)&x_{21}&x_{12}&0\\
			x_{12}&\text{tr}(X)&x_{21}&0\\
			x_{21}&x_{12}&\text{tr}(X)&0\\
			0&0&0&\left[\text{tr}(X)+x_{12}+x_{21}\right]I_{d-3}}
	\end{align*}
	for all $X=[x_{ij}]\in\M{2}$, where $W_j$'s are given by \eqref{eq-CR-d in 2 to d}. Check that
	\begin{align*}
		\Phi(I_2)=\frac{1}{2d}\sum_{j=1}^dW_j^*W_j=\frac{I_{d}}{d},\; \text{ and }\; \Phi^*(I_d)=\frac{1}{2d}\sum_{j=1}^dW_jW_j^*=Z.
	\end{align*}
	Therefore $\Phi\in\msc{CP}(\M{2},\M{d}; Z, \frac{I_{d}}{d})$. By Lemma \ref{lem-CR-d in 2 to d}, it follows that the pair of sets $\{W_i^*W_j\}_{i,j=1}^d\subseteq\M{d}$ and $\{W_jW_i^*\}_{i,j=1}^d\subseteq\M{2}$ is bi-linearly independent. Hence, by Theorem \ref{thm-extr-condition-for doubly}, we conclude that $\Phi$ is an extreme point of the set $\msc{CP}(\M{2},\M{d}; Z, \frac{I_{d}}{d})$ with $CR(\Phi)=d$. 
\end{proof}

\begin{cor}
	$\msc{MR}(Z, \frac{I_{d}}{d})=d=\lfloor{\sqrt{2^2+d^2-1}}\rfloor$ for $d\geq 4$.
\end{cor}

\begin{proof}
	It follows from the Theroerm \ref{thm- CR-d- 2 to d} since $\lfloor{\sqrt{2^2+d^2-1}}\rfloor=d$.
\end{proof}

\begin{center}
	\item \subsection{$(d_1,d_2)=(d,d+1), d\geq 3$.}
\end{center}

In this subsection, we present an example of an extreme point of the set $\msc{CP}(\M{d}, \M{d+1}; \frac{I_d}{d}, \frac{I_{d+1}}{d+1})$ with Choi-rank $d+1$ for $d\geq 2$. Note that the expression $\lfloor{\sqrt{d^2+(d+1)^2-1}}\rfloor$ is equal to $d+1$ if $d\leq 3$, and exceeds $d+1$ if $d\geq 4$. So this example for $d=3$ asserts that the upper bound given by K. R. Parthasarathy is indeed sharp for $(d_1, d_2)=(3,4)$. We begin with the following lemma.

\begin{lem}\label{lem-CR-4 in 3 to 4}
	Let $\{W_j\}_{j=1}^4\subseteq\M{3\times 4}$ be the set of matrices given by
	\begin{align}\label{eq-CR-4 in 3 to 4}
			\begin{cases}
				W_1=\ranko{e_1^{(3)}}{e_1^{(4)}}+\ranko{e_2^{(3)}}{e_2^{(4)}}+\ranko{e_3^{(3)}}{e_3^{(4)}}=[e_1^{(3)},e_2^{(3)},e_3^{(3)},0],\\
				W_2=\ranko{e_1^{(3)}}{e_2^{(4)}}+\ranko{e_2^{(3)}}{e_3^{(4)}}+\ranko{e_3^{(3)}}{e_4^{(4)}}=[0,e_1^{(3)},e_2^{(3)},e_3^{(3)}],\\
				W_3=\ranko{e_3^{(3)}}{e_1^{(4)}}+\ranko{e_1^{(3)}}{e_3^{(4)}}+\ranko{e_2^{(3)}}{e_4^{(4)}}=[e_3^{(3)},0,e_1^{(3)},e_2^{(3)}],\\
				W_4=\ranko{e_2^{(3)}}{e_1^{(4)}}+\ranko{e_3^{(3)}}{e_2^{(4)}}+\ranko{e_1^{(3)}}{e_4^{(4)}}=[e_2^{(3)},e_3^{(3)},0,e_1^{(3)}].
			\end{cases}
    \end{align}
	Then the pair of sets $\{W_i^*W_j:1\leq i,j\leq 4\}\subseteq\M{4}$ and $\{W_jW_i^*:1\leq i,j\leq 4\}\subseteq\M{3}$ is bi-linearly independent.
\end{lem}

\begin{proof}
	
	We compute

	\footnotesize
	\begin{align*}
		\begin{aligned}[c]
			W_1^*W_1&=E_{11}+E_{22}+E_{33},\\
			W_1^*W_2&=E_{12}+E_{23}+E_{34},\\
			W_1^*W_3&=E_{13}+E_{24}+E_{31},\\
			W_1^*W_4&=E_{14}+E_{21}+E_{32},\\
			W_2^*W_1&=E_{21}+E_{32}+E_{43},\\
			W_2^*W_2&=E_{22}+E_{33}+E_{44},\\
			W_2^*W_3&=E_{23}+E_{34}+E_{41},\\
			W_2^*W_4&=E_{24}+E_{31}+E_{42},\\
		\end{aligned}
		\qquad
		\begin{aligned}[c]
			W_3^*W_1&=E_{13}+E_{31}+E_{42},\\
			W_3^*W_2&=E_{14}+E_{32}+E_{43},\\
			W_3^*W_3&=E_{11}+E_{33}+E_{44},\\
			W_3^*W_4&=E_{12}+E_{34}+E_{41},\\
			W_4^*W_1&=E_{12}+E_{23}+E_{41},\\
			W_4^*W_2&=E_{13}+E_{24}+E_{42},\\
			W_4^*W_3&=E_{14}+E_{21}+E_{43},\\
			W_4^*W_4&=E_{11}+E_{22}+E_{44},\\
		\end{aligned}
		\qquad
		 \begin{aligned}[c]
			W_1W_1^*&=I_3,\\
			W_1W_2^*&=F_{21}+F_{32},\\
			W_1W_3^*&=F_{13}+F_{31},\\
			W_1W_4^*&=F_{12}+F_{23},\\
			W_2W_1^*&=F_{12}+F_{23},\\
			W_2W_2^*&=I_3,\\
			W_2W_3^*&=F_{21}+F_{32},\\
			W_2W_4^*&=F_{13}+F_{31},\\
		\end{aligned}
		\qquad
		\begin{aligned}[c]
			W_3W_1^*&=F_{13}+F_{31},\\
			W_3W_2^*&=F_{12}+F_{23},\\
			W_3W_3^*&=I_3,\\
			W_3W_4^*&=F_{21}+F_{32},\\
			W_4W_1^*&=F_{21}+F_{32},\\
			W_4W_2^*&=F_{13}+F_{31},\\
			W_4W_3^*&=F_{12}+F_{23},\\
			W_4W_4^*&=I_3.
		\end{aligned}
	\end{align*}
	\normalsize
	where $F_{ij}:=\ranko{e_i^{(3)}}{e_j^{(3)}}$ for all $1\leq i,j\leq 3$. Let $A=[a_{ij}]\in\M{4}$ such that
	\begin{align*}
		\sum_{i,j=1}^4a_{ij}W_i^*W_j=0=\sum_{i,j=1}^4a_{ij}W_jW_i^*.
	\end{align*}

	A simple calculation yields
	\begin{align*}
		\Matrix{a_{11}+a_{33}+a_{44}&a_{12}+a_{34}+a_{41}&a_{13}+a_{31}+a_{42}&a_{14}+a_{32}+a_{43}\\a_{14}+a_{21}+a_{43}&a_{11}+a_{22}+a_{44}&a_{12}+a_{23}+a_{41}&a_{13}+a_{24}+a_{42}\\a_{13}+a_{24}+a_{31}&a_{14}+a_{21}+a_{32}&a_{11}+a_{22}+a_{33}&a_{12}+a_{23}+a_{34}\\a_{23}+a_{34}+a_{41}&a_{24}+a_{31}+a_{42}&a_{21}+a_{32}+a_{43}&a_{22}+a_{33}+a_{44}}=0
	\end{align*}
	and
	\begin{align*}
		\Matrix{\tr(A)&a_{12}+a_{23}+a_{34}+a_{41}&a_{13}+a_{24}+a_{31}+a_{42}\\a_{14}+a_{21}+a_{32}+a_{43}&\tr(A)&a_{12}+a_{23}+a_{34}+a_{41}\\a_{13}+a_{24}+a_{31}+a_{42}&a_{14}+a_{21}+a_{32}+a_{43}&\tr(A)}=0.
	\end{align*}

	From the above two matrices, we deduce that $a_{ij}=0$ for all $1\leq i,j\leq 4$. This shows that the pair of sets $\{W_i^*W_j:1\leq i,j\leq 4\}\subseteq\M{4}$ and $\{W_jW_i^*:1\leq i,j\leq 4\}\subseteq\M{3}$ is bi-linearly independent.
\end{proof}

\begin{thm}\label{thm-CR 4- from 3 to 4}
	There exists an extreme point of the set $\msc{CP}(\M{3},\M{4}; \frac{I_{3}}{3},\frac{I_{4}}{4})$ with Choi-rank $4$.
\end{thm}

\begin{proof}
	Let $\Phi:\M{3}\to\M{4}$ be a CP map defined by
	\begin{align*}
		\Phi(X)
		&=\frac{1}{12}\sum_{j=1}^{4}W_j^*XW_j\\
		&=\frac{1}{12}\Matrix{\tr(X)&x_{12}+x_{23}&x_{13}+x_{31}&x_{21}+x_{32}\\x_{21}+x_{32}&\tr(X)&x_{23}+x_{12}&x_{13}+x_{31}\\x_{31}+x_{13}&x_{32}+x_{21}&\tr(X)&x_{12}+x_{23}\\x_{12}+x_{23}&x_{13}+x_{31}&x_{21}+x_{32}&\tr(X)}
	\end{align*}
	for all $X=[x_{ij}]\in\M{3}$, where $W_j$'s are given by \eqref{eq-CR-4 in 3 to 4}. Check that
	\begin{align*}
		\Phi(I_3)=\frac{1}{12}\sum_{j=1}^4W_j^*W_j=\frac{I_{4}}{4},\; \text{ and }\; \Phi^*(I_4)=\frac{1}{12}\sum_{j=1}^4W_jW_j^*=\frac{I_3}{3}.
	\end{align*}
	Therefore $\Phi\in\msc{CP}(\M{3},\M{4}; \frac{I_{3}}{3},\frac{I_{4}}{4})$. By Lemma \ref{lem-CR-4 in 3 to 4}, it follows that the pair of sets $\{W_i^*W_j\}_{i,j=1}^4\subseteq\M{4}$ and $\{W_jW_i^*\}_{i,j=1}^4\subseteq\M{3}$ is bi-linearly independent. Hence, by Theorem \ref{thm-extr-condition-for doubly}, we conclude that $\Phi$ is an extreme point of the set $\msc{CP}(\M{3},\M{4}; \frac{I_{3}}{3},\frac{I_{4}}{4})$ with  $CR(\Phi)=4$.
\end{proof}

\begin{cor}
	$\msc{MR}(\frac{I_{3}}{3},\frac{I_{4}}{4})=4=\lfloor{\sqrt{3^2+4^2-1}}\rfloor$.
\end{cor}

\begin{proof}
	Follows from the Theorem \ref{thm-CR 4- from 3 to 4}.
\end{proof}

\begin{lem}\label{lem-CR-d+1 in d to d+1}
	Given $d\geq 2$. Let $\{W_j\}_{j=1}^{d+1}\subseteq\M{d\times (d+1)}$ be the set of matrices given by
	\begin{align}\label{eq-CR-d+1 in d to d+1}
			\begin{cases}
				W_1=[e_1^{(d)},e_2^{(d)},\cdots,e_d^{(d)},0],\\
				W_2=[0,e_1^{(d)},e_2^{(d)},\cdots,e_d^{(d)}],\\
				W_3=[e_d^{(d)},0,e_1^{(d)},\cdots,e_{(d-1)}^{(d)}],\\
				\vdots\\
				W_{d+1}=[e_2^{(d)},e_3^{(d)},\cdots,e_d^{(d)},0,e_1^{(d)}].
			\end{cases}
	\end{align}
	Then the pair of sets $\{W_i^*W_j:1\leq i,j\leq d+1\}\subseteq\M{d+1}$ and $\{W_jW_i^*:1\leq i,j\leq d+1\}\subseteq\M{d}$ is bi-linearly independent.
\end{lem}

\begin{proof}
	The proof is in similar lines to the proof of the Lemma \ref{lem-CR-4 in 3 to 4}.
\end{proof}

\begin{thm}\label{thm-CR d+1-d to d+1}
	Given $d\geq 2$. There exists an extreme point of the set $\msc{CP}(\M{d},\M{d+1}; \frac{I_{d}}{d},\frac{I_{d+1}}{d+1})$ with Choi-rank $d+1$.
\end{thm}

\begin{proof}
	Let $\Phi:\M{d}\to\M{d+1}$ be the map defined by $	\Phi=\frac{1}{d^2+d}\sum_{j=1}^{d+1}\mathrm{Ad}_{W_j}$, where $W_j$'s are given by \eqref{eq-CR-d+1 in d to d+1} for $1\leq j\leq d+1$. Check that $\Phi\in\msc{CP}(\M{d},\M{d+1}; \frac{I_{d}}{d},\frac{I_{d+1}}{d+1})$. By Lemma \ref{lem-CR-d+1 in d to d+1}, it follows that the pair of sets $\{W_i^*W_j\}_{i,j=1}^{d+1}\subseteq\M{d+1}$ and $\{W_jW_i\}_{i,j=1}^{d+1}\subseteq\M{d}$ is bi-linearly independent. Hence, by Theorem \ref{thm-extr-condition-for doubly}, we conclude that $\Phi$ is an extreme point of the set $\msc{CP}(\M{d},\M{d+1}; \frac{I_{d}}{d},\frac{I_{d+1}}{d+1})$ with  $CR(\Phi)=d+1$.
\end{proof}

\begin{cor}
	If $d\geq 2$, then $\msc{MR}(\frac{I_{d}}{d},\frac{I_{d+1}}{d+1})\geq d+1$. Moreover, $\msc{MR}(\frac{I_{d}}{d},\frac{I_{d+1}}{d+1})=d+1$ if $d=2,3$.
\end{cor}

\begin{proof}
	The first part is immediate from the Theorem \ref{thm-CR d+1-d to d+1}. The second part also follows because $\msc{MR}(\frac{I_{d}}{d},\frac{I_{d+1}}{d+1})\leq \lfloor{\sqrt{d^2+(d+1)^2-1}}\rfloor=d+1$ if $d=2,3$.
\end{proof}

We wish to conclude this paper with the following mentioned. Given $d_1=d_2\geq 3$ or $d_1\neq d_2$ with $d_1,d_2\geq 2$, this would be interesting to study whether it is possible to give an algorithm to construct an extreme point that attain the upper bound given by K. R. Parthasarathy. We are not able to answer this at this stage. We leave this for furthur investigation in the future.

\section{Acknowledgement}

The authors are extremely grateful to B. V. Rajarama Bhat for introducing this area of research. The research of the first and second named author is supported by Institute Postdoctoral Fellowship of Indian Institute of Technology Bombay.

\section{Appendix}\label{appendix}

\begin{lem}\label{lem-CR-6-appendix}
	Let $\{W_j:1\leq j\leq 6\}\subseteq\M{5}$ be the set of matrices given by 
    \begin{align}\label{eq-CR-6-appendix}
		\begin{cases}
			W_1&=E_{13}+E_{32}=[0,e_3,e_1,0,0],\\
			W_2&=E_{24}+E_{43}=[0,0,e_4,e_2,0],\\
			W_3&=\sqrt{2}E_{35}+E_{54}=[0,0,0,e_5,\sqrt{2}e_3],\\
			W_4&=E_{14}+E_{42}=[0,e_4,0,e_1,0],\\
			W_5&=E_{15}+E_{41}+E_{53}=[e_4,0,e_5,0,e_1],\\
			W_6&=\sqrt{2}E_{21}+E_{52}=[\sqrt{2}e_2,e_5,0,0,0].
		\end{cases}
	\end{align}
	Then the pair of sets $\{W_i^*W_j:1\leq i,j\leq 6\}$ and $\{W_jW_i^*:1\leq i,j\leq 6\}$ is bi-linearly independent.
\end{lem}

\begin{proof}
We compute
    \footnotesize 
		\begin{align*}
			\begin{aligned}[c]
				W_1^*W_1&=E_{33}+E_{22},\\
				W_1^*W_2&=0,\\
				W_1^*W_3&=\sqrt{2}E_{25},\\
				W_1^*W_4&=E_{34},\\
				W_1^*W_5&=E_{35},\\
				W_1^*W_6&=0,\\
				W_2^*W_1&=0,\\
				W_2^*W_2&=E_{33}+E_{44},\\
				W_2^*W_3&=0,\\
				W_2^*W_4&=E_{32},\\
				W_2^*W_5&=E_{31},\\
				W_2^*W_6&=\sqrt{2}E_{41},\\
				W_3^*W_1&=\sqrt{2}E_{52},\\
				W_3^*W_2&=0,\\
				W_3^*W_3&=E_{44}+2E_{55},\\
				W_3^*W_4&=0,\\
				W_3^*W_5&=E_{43},\\
				W_3^*W_6&=E_{42},
			\end{aligned}
			\qquad 
			\begin{aligned}[c]
				W_4^*W_1&=E_{43},\\
				W_4^*W_2&=E_{23},\\
				W_4^*W_3&=0,\\
				W_4^*W_4&=E_{22}+E_{44},\\
				W_4^*W_5&=E_{21}+E_{45},\\
				W_4^*W_6&=0,\\
				W_5^*W_1&=E_{53},\\
				W_5^*W_2&=E_{13},\\
				W_5^*W_3&=E_{34},\\
				W_5^*W_4&=E_{12}+E_{54},\\
				W_5^*W_5&=E_{11}+E_{33}+E_{55},\\
				W_5^*W_6&=E_{32},\\
				W_6^*W_1&=0,\\
				W_6^*W_2&=\sqrt{2}E_{14},\\
				W_6^*W_3&=E_{24},\\
				W_6^*W_4&=0,\\
				W_6^*W_5&=E_{23},\\
				W_6^*W_6&=2E_{11}+E_{22},\\
			\end{aligned}
			\qquad
			\begin{aligned}[c]
				   W_1W_1^*&=E_{11}+E_{33},\\
					W_1W_2^*&=E_{14},\\
					W_1W_3^*&=0,\\
					W_1W_4^*&=E_{34},\\
					W_1W_5^*&=E_{15},\\
					W_1W_6^*&=E_{35},\\
					W_2W_1^*&=E_{41},\\
					W_2W_2^*&=E_{22}+E_{44},\\
					W_2W_3^*&=E_{25},\\
					W_2W_4^*&=E_{21},\\
					W_2W_5^*&=E_{45},\\
					W_2W_6^*&=0,\\
					W_3W_1^*&=0,\\
					W_3W_2^*&=E_{52},\\
					W_3W_3^*&=2E_{33}+E_{55},\\
					W_3W_4^*&=E_{51},\\
					W_3W_5^*&=\sqrt{2}E_{31},\\
					W_3W_6^*&=0,\\
			\end{aligned}
			\qquad 
			\begin{aligned}[c]
					W_4W_1^*&=E_{43},\\
					W_4W_2^*&=E_{12},\\
					W_4W_3^*&=E_{15},\\
					W_4W_4^*&=E_{11}+E_{44},\\
					W_4W_5^*&=0,\\
					W_4W_6^*&=E_{45},\\
					W_5W_1^*&=E_{51},\\
					W_5W_2^*&=E_{54},\\
					W_5W_3^*&=\sqrt{2}E_{13},\\
					W_5W_4^*&=0\\
					W_5W_5^*&=E_{11}+E_{44}+E_{55},\\
					W_5W_6^*&=\sqrt{2}E_{42},\\
					W_6W_1^*&=E_{53},\\
					W_6W_2^*&=0,\\
					W_6W_3^*&=0,\\
					W_6W_4^*&=E_{54},\\
					W_6W_5^*&=\sqrt{2}E_{24},\\
					W_6W_6^*&=2E_{22}+E_{55}.
				\end{aligned}
			\end{align*}
	\normalsize
	
	Let $A=[a_{ij}]\in\M{5}$ such that
	\begin{align*}
		\sum_{i,j=1}^6a_{ij}W_i^*W_j=0=\sum_{i,j=1}^6a_{ij}W_jW_i^*.
	\end{align*}

	A simple calculation yields
    \small
	\begin{align*}
		\Matrix{a_{55}+2a_{66}&a_{54}&a_{52}&\sqrt{2}a_{62}&0\\
			a_{45}&a_{11}+a_{44}+a_{66}&a_{42}+a_{65}&a_{63}&\sqrt{2}a_{13}\\
			a_{25}&a_{24}+a_{56}&a_{11}+a_{22}+a_{55}&a_{14}+a_{53}&a_{15}\\
			\sqrt{2}a_{26}&a_{36}&a_{35}+a_{41}&a_{22}+a_{33}+a_{44}&a_{45}\\
			0&\sqrt{2}a_{31}&a_{51}&a_{54}&2a_{33}+a_{55}}=0
	\end{align*}
    \normalsize
	and 
    \small
	\begin{align*}
		\Matrix{a_{11}+a_{44}+a_{55}&a_{24}&\sqrt{2}a_{35}&a_{21}&a_{34}+a_{51}\\
			a_{42}&a_{22}+2a_{66}&0&\sqrt{2}a_{56}&a_{32}\\
			\sqrt{2}a_{53}&0&a_{11}+2a_{33}&a_{41}&a_{61}\\
			a_{12}&\sqrt{2}a_{65}&a_{14}&a_{22}+a_{44}+a_{55}&a_{52}+a_{64}\\
			a_{15}+a_{43}&a_{23}&a_{16}&a_{25}+a_{46}&a_{33}+a_{55}+a_{66}}=0.
	\end{align*}
    \normalsize

	This implies that $a_{ij}=0$ for all $1\leq i\neq j\leq 6$. Furthermore, we have $Bx=0$, where $x=(a_{11},a_{22},\hdots, a_{66})^{\T}\in\mbb{C}^6$ and $B\in\M{10\times 6}$ given by
    \small
	\begin{align*}
		B^{\T}:=\Matrix{0&1&1&0&0&1&0&1&0&0\\0&0&1&1&0&0&1&0&1&0\\0&0&0&1&2&0&0&2&0&1\\0&1&0&1&0&1&0&0&1&0\\1&0&1&0&1&1&0&0&1&1\\2&1&0&0&0&0&2&0&0&1}.
	\end{align*}
    \normalsize
    
    Check that rank$(B)=6$. So the null space of $B$ is trivial. Therefore, $x=0$ and hence $a_{ij}=0$ for all $1\leq i,j\leq 6$. This shows that the pair of sets $\{W_i^*W_j:1\leq i,j\leq 6\}$ and $\{W_jW_i^*:1\leq i,j\leq 6\}$ is bi-linearly independent.
\end{proof}

 \begin{lem}\label{lem-CR-7-appendix}
	Let $\{W_j:1\leq j\leq 7\}\subseteq\M{5}$ be the set of matrices given by
	\begin{align}\label{eq-CR-7-appendix}
		\begin{cases}
			W_1&=\sqrt{2}E_{13}+E_{32}+E_{54}=[0,e_3,\sqrt{2}e_1,e_5,0],\\
			W_2&=E_{24}+E_{43}=[0,0,e_4,e_2,0],\\
			W_3&=\sqrt{2}E_{35}+\sqrt{3}E_{54}=[0,0,0,\sqrt{3}e_5,\sqrt{2}e_3],\\
			W_4&=E_{14}+E_{42}=[0,e_4,0,e_1,0]\\
			W_5&=E_{15}+2E_{41}+E_{53}=[2e_4,0,e_5,0,e_1],\\
			W_6&=\sqrt{2}E_{21}+E_{52}=[\sqrt{2}e_2,e_5,0,0,0],\\
			W_7&=\sqrt{2}E_{13}+\sqrt{3}iE_{32}+\sqrt{3}E_{25}=[0,\sqrt{3}ie_3,\sqrt{2}e_1,0,\sqrt{3}e_2].
		\end{cases}
	\end{align}
	Then the pair of sets $\{W_i^*W_j:1\leq i,j\leq 7\}$ and $\{W_jW_i^*: 1\leq i,j\leq 7\}$ is bi-linearly independent
\end{lem}

\begin{proof}
	We first compute $W_i^*W_j$ for all $1\leq i,j\leq 7$ as listed below:
	
	\footnotesize
	\begin{align*}
		\begin{aligned}[c]
			W_1^*W_1&=E_{22}+2E_{33}+E_{44},\\
			W_1^*W_2&=0,\\
			W_1^*W_3&=\sqrt{2}E_{25}+\sqrt{3}E_{44},\\
			W_1^*W_4&=\sqrt{2}E_{34},\\
			W_1^*W_5&=\sqrt{2}E_{35}+E_{43},\\
			W_1^*W_6&=E_{42},\\
			W_1^*W_7&=\sqrt{3}iE_{22}+2E_{33},\\
		\end{aligned}
		\qquad
		\begin{aligned}[c]
			W_2^*W_1&=0,\\
			W_2^*W_2&=E_{33}+E_{44},\\
			W_2^*W_3&=0,\\
			W_2^*W_4&=E_{32},\\
			W_2^*W_5&=2E_{31},\\
			W_2^*W_6&=\sqrt{2}E_{41},\\
			W_2^*W_7&=\sqrt{3}E_{45},\\
		\end{aligned}
		\qquad 
		\begin{aligned}[c]
			W_3^*W_1&=\sqrt{2}E_{52}+\sqrt{3}E_{44},\\
			W_3^*W_2&=0,\\
			W_3^*W_3&=3E_{44}+2E_{55},\\
			W_3^*W_4&=0,\\
			W_3^*W_5&=\sqrt{3}E_{43},\\
			W_3^*W_6&=\sqrt{3}E_{42},\\
			W_3^*W_7&=\sqrt{6}iE_{52},\\
			\end{aligned}
		\qquad 
		\begin{aligned}[c]
			W_4^*W_1&=\sqrt{2}E_{43},\\
			W_4^*W_2&=E_{23},\\
			W_4^*W_3&=0,\\
			W_4^*W_4&=E_{22}+E_{44}\\
			W_4^*W_5&=2E_{21}+E_{45},\\
			W_4^*W_6&=0,\\
			W_4^*W_7&=\sqrt{2}E_{43},\\
			\end{aligned}
	\end{align*}

	\begin{align*}
		\begin{aligned}[c]
			W_5^*W_1&=\sqrt{2}E_{53}+E_{34},\\
			W_5^*W_2&=2E_{13},\\
			W_5^*W_3&=\sqrt{3}E_{34},\\
			W_5^*W_4&=2E_{12}+E_{54},\\
			W_5^*W_5&=4E_{11}+E_{33}+E_{55},\\
			W_5^*W_6&=E_{32},\\
			W_5^*W_7&=\sqrt{2}E_{53},\\
		\end{aligned}
		\qquad
		\begin{aligned}[c]
			W_6^*W_1&=E_{24},\\
			W_6^*W_2&=\sqrt{2}E_{14},\\
			W_6^*W_3&=\sqrt{3}E_{24},\\
			W_6^*W_4&=0,\\
			W_6^*W_5&=E_{23},\\
			W_6^*W_6&=2E_{11}+E_{22},\\
			W_6^*W_7&=\sqrt{6}E_{15},\\
		\end{aligned}
		\qquad
		\begin{aligned}[c]
			W_7^*W_1&=-\sqrt{3}iE_{22}+2E_{33},\\
			W_7^*W_2&=\sqrt{3}E_{54},\\
			W_7^*W_3&=-\sqrt{6}iE_{25},\\
			W_7^*W_4&=\sqrt{2}E_{34},\\
			W_7^*W_5&=\sqrt{2}E_{35},\\
			W_7^*W_6&=\sqrt{6}E_{51},\\
			W_7^*W_7&=3E_{22}+2E_{33}+3E_{55}.\\
		\end{aligned}
	\end{align*}
	\normalsize

	Next, we compute $W_jW_i^*$ for all $1\leq i,j\leq 7$ as listed below:

	\footnotesize
	\begin{align*}
		\begin{aligned}[c]
			W_1W_1^*&=2E_{11}+E_{33}+E_{55},\\
			W_1W_2^*&=\sqrt{2}E_{14}+E_{52},\\
			W_1W_3^*&=\sqrt{3}E_{55},\\
			W_1W_4^*&=E_{51}+E_{34},\\
			W_1W_5^*&=\sqrt{2}E_{15},\\
			W_1W_6^*&=E_{35},\\
			W_1W_7^*&=2E_{11}-\sqrt{3}iE_{33},\\
		\end{aligned}
		\qquad
		\begin{aligned}[c]
			W_2W_1^*&=E_{25}+\sqrt{2}E_{41},\\
			W_2W_2^*&=E_{22}+E_{44},\\
			W_2W_3^*&=\sqrt{3}E_{25},\\
			W_2W_4^*&=E_{21},\\
			W_2W_5^*&=E_{45},\\
			W_2W_6^*&=0,\\
			W_2W_7^*&=\sqrt{2}E_{41},\\
		\end{aligned}
		\qquad 
		\begin{aligned}[c]
			W_3W_1^*&=\sqrt{3}E_{55},\\
			W_3W_2^*&=\sqrt{3}E_{52},\\
			W_3W_3^*&=2E_{33}+3E_{55},\\
			W_3W_4^*&=\sqrt{3}E_{51},\\
			W_3W_5^*&=\sqrt{2}E_{31},\\
			W_3W_6^*&=0,\\
			W_3W_7^*&=\sqrt{6}E_{32},\\
		\end{aligned}
		\qquad 
		\begin{aligned}[c]
			W_4W_1^*&=E_{15}+E_{43},\\
			W_4W_2^*&=E_{12},\\
			W_4W_3^*&=\sqrt{3}E_{15},\\
			W_4W_4^*&=E_{11}+E_{44},\\
			W_4W_5^*&=0,\\
			W_4W_6^*&=E_{45},\\
			W_4W_7^*&=-\sqrt{3}iE_{43},\\
		\end{aligned}
	\end{align*}

	\begin{align*}
		\begin{aligned}[c]
			W_5W_1^*&=\sqrt{2}E_{51},\\
			W_5W_2^*&=E_{54},\\
			W_5W_3^*&=\sqrt{2}E_{13},\\
			W_5W_4^*&=0,\\
			W_5W_5^*&=E_{11}+4E_{44}+E_{55},\\
			W_5W_6^*&=2\sqrt{2}E_{42},\\
			W_5W_7^*&=\sqrt{3}E_{12}+\sqrt{2}E_{51},\\
		\end{aligned}
		\qquad
		\begin{aligned}[c]
			W_6W_1^*&=E_{53},\\
			W_6W_2^*&=0,\\
			W_6W_3^*&=0,\\
			W_6W_4^*&=E_{54},\\
			W_6W_5^*&=2\sqrt{2}E_{24},\\
			W_6W_6^*&=2E_{22}+E_{55},\\
			W_6W_7^*&=-\sqrt{3}iE_{53},\\
		\end{aligned}
		\qquad
		\begin{aligned}[c]
			W_7W_1^*&=2E_{11}+\sqrt{3}iE_{33},\\
			W_7W_2^*&=\sqrt{2}E_{14},\\
			W_7W_3^*&=\sqrt{6}E_{23},\\
			W_7W_4^*&=\sqrt{3}iE_{34},\\
			W_7W_5^*&=\sqrt{2}E_{15}+\sqrt{3}E_{21},\\
			W_7W_6^*&=\sqrt{3}iE_{35},\\
			W_7W_7^*&=2E_{11}+3E_{22}+3E_{33}.\\
		\end{aligned}
	\end{align*}
	\normalsize

	Let $A=[a_{ij}]\in\M{5}$ such that $\sum_{i,j=1}^7a_{ij}W_i^*W_j=0=\sum_{i,j=1}^7a_{ij}W_jW_i^*$. A simple calculation yields

	\resizebox{.98\linewidth}{!}{
		\begin{minipage}{\linewidth}
			\begin{align*}
				&\Matrix{4a_{55}+2a_{66}&2a_{54}&2a_{52}&\sqrt{2}a_{62}&\sqrt{6}a_{67}\\
					2a_{45}&\sqrt{3}i(a_{17}-a_{71})+a_{11}+a_{44}+a_{66}+3a_{77}&a_{42}+a_{65}&a_{61}+\sqrt{3}a_{63}&\sqrt{2}a_{13}-\sqrt{6}ia_{73}\\
					2a_{25}&a_{24}+a_{56}&2(a_{17}+a_{71})+2a_{11}+a_{22}+a_{55}+2a_{77}&\sqrt{2}a_{14}+a_{51}+\sqrt{3}a_{53}+\sqrt{2}a_{74}&\sqrt{2}(a_{15}+a_{75})\\
					\sqrt{2}a_{26}&a_{16}+\sqrt{3}a_{36}&a_{15}+\sqrt{3}a_{35}+\sqrt{2}(a_{41}+a_{47})&\sqrt{3}(a_{13}+a_{31})+a_{11}+a_{22}+3a_{33}+a_{44}&\sqrt{3}a_{27}+a_{45}\\	\sqrt{6}a_{76}&\sqrt{2}a_{31}+\sqrt{6}ia_{37}&\sqrt{2}(a_{51}+a_{57})&a_{54}+\sqrt{3}a_{72}&2a_{33}+a_{55}+3a_{77}}=0,\\
				&\Matrix{2(a_{17}+a_{71})+2a_{11}+a_{44}+a_{55}+2a_{77}&a_{24}+\sqrt{3}a_{75}&\sqrt{2}a_{35}&\sqrt{2}(a_{21}+a_{27})&a_{14}+\sqrt{3}a_{34}+\sqrt{2}a_{51}+a_{57}\\
					a_{42}+\sqrt{3}a_{57}&a_{22}+2a_{66}+3a_{77}&\sqrt{6}a_{37}&2\sqrt{2}a_{56}&a_{12}+\sqrt{3}a_{32}\\
					\sqrt{2}a_{53}&\sqrt{6}a_{73}&\sqrt{3}i(a_{17}-a_{71})+a_{11}+2a_{33}+3a_{77}&a_{41}+\sqrt{3}ia_{47}&a_{61}+\sqrt{3}ia_{67}\\
					\sqrt{2}(a_{12}+a_{72})&2\sqrt{2}a_{65}&a_{14}-\sqrt{3}ia_{74}&a_{22}+a_{44}+4a_{55}&a_{52}+a_{64}\\
					\sqrt{2}(a_{15}+a_{75})+a_{41}+\sqrt{3}a_{43}&a_{21}+\sqrt{3}a_{23}&a_{16}-\sqrt{3}ia_{76}&a_{25}+a_{46}&\sqrt{3}(a_{13}+a_{31})+a_{11}+3a_{33}+a_{55}+a_{66}}=0.
			\end{align*}
		\end{minipage}
	}
	
	This implies that $a_{ij}=0$ for all $1\leq i\neq j\leq 7$ except $a_{17}$ and $a_{71}$. Furthermore, we have $Bx=0$, where $x=(a_{11},a_{22},\hdots,a_{77},a_{17},a_{71})^{\T}\in\mbb{C}^9$ and $B\in\M{10\times 9}$ is given by
	\footnotesize
	\begin{align*}
		B:=\Matrix{2&1&0&0&1&0&2&2&2\\2&0&0&1&1&0&2&2&2\\1&0&0&1&0&1&3&\sqrt{3}i&-\sqrt{3}i\\1&0&2&0&0&0&3&\sqrt{3}i&-\sqrt{3}i\\1&1&3&1&0&0&0&0&0\\1&0&3&0&1&1&0&0&0\\0&1&0&1&4&0&0&0&0\\0&1&0&0&0&2&3&0&0\\0&0&2&0&1&0&3&0&0\\0&0&0&0&4&2&0&0&0}.
	\end{align*}
	\normalsize

	Check that rank$(B)=9$. So the null space of $B$ is trivial. Therefore, $x=0$ and hence $a_{ij}=0$ for all $1\leq i,j\leq 7$. This shows that the pair of sets $\{W_i^*W_j:1\leq i,j\leq 7\}$ and $\{W_jW_i^*:1\leq i,j\leq 7\}$ is bi-linearly independent. 
\end{proof}

\bibliographystyle{acm}
\bibliography{Bibliography}	

 \end{document}